\documentclass[12pt]{article}
\usepackage{amsmath,amsthm,amssymb,amsfonts}
\usepackage{cite}
\usepackage{xcolor}
\usepackage{comment}
\usepackage{bm} 
\usepackage{float}
\usepackage[pdftex]{hyperref}
\usepackage{units}
\usepackage{graphicx}
\usepackage{multirow} 
\usepackage[normalem]{ulem}
\usepackage{tabularx}
\allowdisplaybreaks

\newcommand\Ber{{\rm Ber}}
\newcommand\eqlaw{{\buildrel d \over = }}
\newcommand\Poi{{\rm Poi}}

\newcommand{\field}{{\mathbb F}}
\newtheorem{theorem}{Theorem}[section]
\newtheorem{lemma}{Lemma}[section]

\newtheorem{cor}{Corollary}[section]
\newtheorem{remark}{Remark}[section]
\newtheorem{example}{Example}[section]
\newcommand\polya{{P\' olya}}
\newcommand\convD{{\buildrel {\cal D} \over \longrightarrow}}
\newcommand\given{\, \vert \, } 
 
\newcommand\almostsure{\buildrel a.s. \over  \longrightarrow}

\newcommand\Prob{{\mathbb P}}
\newcommand\E{{\mathbb E}}
\newcommand\V{{\mathbb V}{\rm ar}}

\newcommand\normal{{\cal N}}
\newcommand\indicator{{\mathbb I}}
\begin{document}
\begin{center}
{\bf \Huge Triggered urn models for frequently asked questions (FAQ)

\bigskip
\large
Irene Crimaldi\footnote{IMT School for Advanced Studies Lucca, Piazza San Ponziano 6, 55100 Lucca, Italy; Email: irene.crimaldi@imtlucca.it} \quad
Andrea Ghiglietti\footnote{Department of Statistics and Quantitative Methods, Università degli Studi di Milano-Bicocca, Piazza dell’Ateneo
Nuovo 1, 20126 Milano, Italy; Email: \linebreak andrea.ghiglietti@unimib.it} \quad
Leen Hatem\footnote{Department of Mechanical and Electrical Engineering, Damascus University, Damascus, Syria; Email: Leenhatem00@gmail.com} \quad
Hosam Mahmoud\footnote{Department of Statistics, The George Washington University, Washington, D.C. 20052, U.S.A.; Email: hosam@gwu.edu}

\bigskip
\today}
\end{center}
\begin{abstract}
 \fontsize{11pt}{11pt}\selectfont
{We investigate a nonclassic urn model with triggers that increase the number of colors. The scheme has emerged as a model for web services that set up frequently asked questions (FAQ). 

We present a thorough  asymptotic analysis of the FAQ urn scheme in generality that covers a large number of special
cases, such as Simon's urn. For instance, we consider time dependent triggering probabilities. We identify regularity  conditions on these probabilities that classify the schemes into those where the number of colors in the urn remains almost surely finite or increases to infinity and conditions that tell us whether all the existing colors are observed infinitely often or not. We determine the rank curve, too. 

In view 
of the broad generality of the trigger probabilities, a spectrum of limit distributions appears, from central limit theorems to Poisson approximation, to power-laws, revealing connections to  Heaps exponent and Zipf's law. 

A combinatorial approach to the Simon urn is presented 
to indicate the possibility of such exact analysis, 
which is important for short-term predictions. Extensive simulations on real datasets (from Amazon sales) as well as computer-generated data clearly indicate that the asymptotic and exact theory developed agrees with practice.   }
\end{abstract}
\bigskip\noindent
{\bf Keywords:}  FAQ, database, urn, probabilistic analysis,
Gaussian law, Poisson approximation.

\bigskip\noindent
{\bf MSC:} {60F05, 
                           60C05, 
                           60G46}  
\section{Introduction}
Urn schemes are versatile modeling tools, as they capture the dynamics of
many real-world scenarios. For example, \polya\ urns have received 
phenomenal attention and have become a smashing success in many diverse applications.
The book~\cite{Mah2009} 
presents numerous applications in informatics and biosciences; the survey \cite{Bala} gives a wider perspective and goes over urns that are not of the \polya\ type.

Modern applications keep suggesting newer urn models. In this manuscript, we investigate
a non-classic urn model with occasional triggers. The model has emerged from some considerations in databases. Indeed, there has recently been a growing interest in analyzing FAQ datasets (see~\cite{Campos, DeBruyn, Din, Sumikawa, wan-mcauley, Zhang}).  Yet, the study of the dynamics governing the arrival of new questions in a FAQ database is a novelty of sorts.

The triggered urn scheme proposed is not a standard \polya\ urn model as the number of colors  (questions) in it is increasing. A new inquiry into the database may be viewed as a novelty, which connects this investigation to a large
 body of research on innovation and novelties, wherein a variety of triggered urns 
 are used as models, addressing various types of applications but different from ours   (see~\cite{Chung, ale-cri, Tria1, Tria2, Tria3}). From the theoretical point of view, the model in \cite{Chung} is a special case of the urn scheme we are setting up in this manuscript for FAQ databases. Moreover, the main difference between our model and the models in 
\cite{ale-cri,Tria1,Tria2,Tria3} lies in the definition of the probability of having an innovation. 
However, the proper choice of parameters for our general model allows us to find the same  
asymptotic behaviors for the observed number of colors and for the color frequency-rank curve of those models.

The rest of the paper is organized in sections. 
In Subsection~\ref{Subsec:database},
we discuss a common application related to service websites that set up databases for 
``frequently asked questions'' (FAQ).
The formal setup of the triggered urn underlying FAQ databases is given in Section~\ref{Sec:FAQ}, which has two subsections going over the rationale behind the 
dynamics of the urn and an illustrative
example. Section~\ref{Sec:analysis} is on the asymptotic probabilistic analysis of the FAQ triggered urn scheme, with 
three subsections ramifying the analysis into {theorems for the asymptotic behaviors and effective heuristics for the stationary color distribution, and the frequency rank curve.}  
A combinatorial approach to the FAQ urn, with fixed trigger probability across time, is discussed in Section~\ref{Sec:Simon}, allowing for calculations after a small-to-medium
number of time steps, which is important for the short-term revisions of the FAQ database. 
A subsection is dedicated to a brief discussion of dynamic variations of the Simon urn. In Section~\ref{Sec:conc},
we discuss applications to real datasets from Amazon. {Appendix A collects some technical results, and  Appendix B (in two parts) reports on extensive simulations on real datasets  
 that indicate that the theory developed agrees with practice with a fast rate of convergence.}  
\subsection{A database for inquiries}
\label{Subsec:database}
Many public services set up a database of frequently asked questions (FAQ), and provide answers 
as guidance to the public. For example, the State Department of the United States has a 
website named Education USA. The site can be accessed at 

\bigskip
https://educationusa.state.gov/experience-studying-usa/
\\us-educational-system/frequently-asked-questions-faqs

\bigskip
The site provides 37 question/answer (Q/A) pairs, the  first two of which are
 
 \bigskip
Q: There are so many schools in the U.S. How do I decide which schools to apply to?

\medskip
A: Research your options and define your priorities. Contact the Education USA advising center nearest you and browse college search engines online. Check to see if the schools 
you are considering are accredited.

\bigskip
Q: What’s the difference between a college and a university?

\medskip
A: Colleges offer only undergraduate degrees while universities offer graduate degrees as well, but the terms are often used interchangeably.

\bigskip
Some FAQ databases grow as the questions are asked. The first question establishes the database.
The question may repeat for a while, till a new question is posed, at which point 
the FAQ maintenance team adds the question to the FAQ database and provides an answer. The process goes
on, with the questions repeating. When a new question is asked, an additional Q/A pair
is inserted at the bottom of the FAQ database.
 
Important FAQ databases are sanitized. For example,
the Education USA site mentioned does not have repeated questions. For modern FAQ
databases, the AI in the backdrop vets the 
incoming questions. A question presented does not have to be an exact word-for-word 
duplicate of an existing question to be considered the same. The AI considers a question the same
as an existing one, if it has sufficient similarity, such as agreeing on a high percentage
of keywords in the question.
  
With future revisions in mind, the database maintenance team wishes to keep track of how many times any particular question is asked. The purpose is to revise and improve the database, by, 
for example, putting the most frequent questions near the top of the database, to ease the search by the public, or writing more accurate answers to them after further research on the issues. 
  
We can anticipate an observation/evaluation period during which questions are collected, 
so as to revise the database for faster Q/A retrieval or better appearance at a future time. 
\section {The FAQ triggered urn model}
\label{Sec:FAQ}
In this section, we construct a FAQ triggered urn. 
Questions asked are represented by colored balls in an urn,  each color corresponds to
a question and the number of balls of a color reflects in some way the frequency 
or importance of the corresponding question. Initially, the urn starts empty at time 0.
The first trigger is always a  success, adding balls of color 1   
to represent the first question. Later questions appear upon successful 
 triggering events that change the color composition of the urn. So, later questions are represented 
by additional colors (one new color for each new question).
 
Formally, the urn starts empty (at time $n=0$). 
Set $B_0=1$ and let $(B_n)_{n\ge 1}$ be a sequence of independent Bernoulli random variables  with success probabilities 
$p_n\in (0,1)$. Think of the event  
$B_{n-1}=1$ as a trigger at time $n-1$ that
causes a kaleidoscopic change at time $n$.
Let $F$ be a strictly increasing nonnegative 
function defined on~$[1,\infty$) 
with $F(1)>0$. We refer to this function as the {\em update function} of the model. 

{In the context of balls and urns, it is natural to take the update function $F$ as an integer-valued function. However,
toward broader generality, we take it as a real function. Via this extension, each ball color has a ``weight.'' 
This only affects the probability of drawing
from a particular color, which is the proportion of weight allocated to that color with respect to the total weight in the urn.
}

The model works as follows: 
At time $n\ge 1$, the urn evolves: 
\begin{itemize}
\item if $B_{n-1}=1$, a triggering event occurs:  We add to the urn $F(1)$ balls of a new color (i.e.~not already present in the urn), and we register this new color in the {\em history sequence} $\mathcal{S}$ (note that we label the colors according to the order of their appearance: color $1$ is the first to appear, color $2$ is the second to appear and so on). 
\item if $B_{n-1}=0$ (no triggering), proceed normally: We extract one ball from the urn and
	we register its color in $\cal S$. Suppose the color of the ball drawn  is $c$, 
	the number of balls of color $c$ in the urn becomes $F(K_{n,c})$, where 
	$K_{n,c}$ is {the number of times color $c$ has been registered in $\mathcal S$}
    by time $n$ (included). 
	\end{itemize}

 The number of colors in this urn is growing, and it is not a standard \polya\ urn. 
Another essential difference from \polya\ urns is that in a \polya\ urn
scheme, there is   a drawing of a ball (or a group of balls) 
 at {\em each} epoch of time.
In the FAQ urn scheme, at any epoch a Bernoulli random variable is generated,
and a ball drawing is affected {\em only} in the case of failure (when the generated 
Bernoulli random variable is 0). In the case of success  (a triggering event occurs, when the generated Bernoulli random variable is 1), no ball drawing takes place.

Perhaps the closest urn scheme to the present investigation is the Simon model~\cite{Simon}, in which $p_n=p$ is fixed and $F(x) =x$, for all $x\ge 1$. In this special case, if $B_n=0$, 
we sample a ball from the urn and register its color in the history sequence $\mathcal S$. 
Suppose the color of the ball sampled at time $n$ is $c$, and the last time it appeared in the history sequence is $1\le n'<n$.
With the choice of the update 
function, we have $F(K_{n,c}) - F(K_{n',c}) = K_{n,c} - K_{n',c}  =1$, so at time $n$ the number 
of balls of the sampled color is increased by 1, that is, we add one ball of the same color. 
Alternatively,
we have the triggering event $B_n=1$ and we add $F(1)=1$ ball of a new color. In either case, exactly 
one ball is added. 

At any point in time, in Simon's urn scheme, the number of appearances of a color in 
$\mathcal S$
is inductively the number of balls of that color in the urn.
\subsection{Rationale underlying the FAQ urn} 
\label{Subsec:rationale}
There is a motivation for the choices of the dynamics of the FAQ urn. The choice $p_n$, with possibly different $p$'s,
keeps the model general and adaptable to a variety of FAQ applications. For instance,
in some applications, the sequence $p_n$ may be {decreasing}, 
as in a FAQ database with
fundamental questions that most people would ask of  a passport-issuing agency, such as 
``What is the address of the local branch?'' 
One expects that inquiries about the fundamental questions 
to be made a lot. As other less popular questions are inserted in the FAQ 
database, there will be a propensity for the first few fundamental ones. The probability
to pull out an early question from the existing database will increase.

A FAQ database supporting a service related to fashion may experience the opposite
evolution, with {increasing} $p_n$. 
For example, when a fashion house like Channel comes up with a new 
perfume in 2026 or when Mercedes company releases the new 2026 model, there may be interest in certain questions at the beginning of the year 
about the product, such as inquiries about the price. 
Such interest will wane toward the end of the year.  The probability
to pull out a question from the existing database will decrease, as consumers will show more
interest in other current products.

Note that the Matthew principle underlies the FAQ triggered urn---an economics principle according to which ``success breeds success'' or ``the rich get richer.'' In the 
FAQ model, questions asked frequently are favored as they 
are represented by more balls in the urn, and thus
are more likely to come up in the drawings in case of no triggering. 

The assumption that the variables $B_n$ are independent is reasonable for a FAQ database 
serving a large population. The probabilities associated with the Bernoulli variables
may have variations depending on the service.
\subsection{An illustrative example}
Consider Simon's urn as the FAQ triggered urn.
Recall from the discussion right before Subsection~\ref{Subsec:rationale} that in this model
at every time step one ball is added.
Let us take the instance of $\cal S$ in which the initial segment is
 $1121223$ representing a string of  seven
happenings at the beginning of the stochastic path. The  
first trigger $B_0$ always succeeds. 
At  time 1, we have a question, which we call ${\cal Q}_1$,    
 a ball of color 1 is deposited in the urn. After failure of the trigger $B_1$, the question~${\cal Q}_1$ is posed again at time 2, an additional ball of color 1 is added 
to the urn. Right before time 3, there is a successful trigger and a new question is asked; 
it is considered as~${\cal Q}_2$  and is  
coded with a ball of color 2 in the urn. Another trigger failure occurs at $B_3$, the 
question~${\cal Q}_1$ is asked again; another ball of color~1 is added to the urn. 
Then two trigger failures follow and~${\cal Q}_2$ is asked twice in a row, with two balls of color~2---accounting for this happening---added to the urn. Right before time 7,  the trigger $B_6$
succeeds yet again and a new question is posed. It
is considered as~${\cal Q}_3$, which is  coded with a ball of color 3 in the urn. By time 7, we have
three balls of color 1, three balls of color 2, and one ball of color 3, matching the history
and all the frequencies of the questions asked.
The probability of this sequence is 
$$1  \times (1-p) \times p \times \frac 2 3\, (1-p)
\times \frac 1 4 (1-p) \times \frac 2 5 (1-p) \times p = \frac 1 {15}\, p^2 (1-p)^4.$$ 

\section{Analysis of the FAQ triggered urn scheme}
\label{Sec:analysis}
In the sequel, we use the notation $\Ber(s)$ for a Bernoulli random variable 
with success probability $s$, and $\normal (0,1)$ stands for a standard normal random variate.
Let us set ${\cal S}=(S_n)_{n\geq 1}$ and denote by $C_n$ the number of different colors in the urn at time $n$ and registered in $\mathcal S$.  According to the dynamics described in Section~\ref{Sec:FAQ}, we have $C_0=0$ (the urn is initially empty) and
\begin{equation}
C_n=\sum_{i=0}^{n-1} B_i;
\label{Eq:Cn}
\end{equation} 

Moreover, for each epoch $n\ge 1$,
we have
\begin{equation}\label{model-equations}
\begin{split}
	\Prob\big(S_{n}\not \in {\mathcal C}_{n-1}\big) &=p_{n-1}, \qquad\hbox{(with $p_0=1$)};\\
    \Prob(S_{n}=c\,|\, S_1,\dots, S_{n-1}) &=
(1-p_{n-1})\, 
\frac{F(K_{n-1,c})}{\sum_{c'\in \mathcal{C}_{n-1}} F(K_{n-1,c'})},
\;\forall c\in {\mathcal C}_{n-1}\,,
\end{split}
\end{equation}
where ${\mathcal C}_{n-1}=\{1,\dots,C_{n-1}\}$ denotes the set of different colors 
registered in~$\mathcal S$ and present in the urn at  time $n-1$.  Note that the event 
$S_n \not \in {\mathcal C}_{n-1}$ means that {at time $n$ a new color appears in the urn. }

By L\'evy's extension of Borel-Cantelli lemmas (see, for instance, Section 12.15 in~\cite{Williams}) applied to the sequence $(B_n)_{n=0}^\infty$, we get the following.
\begin{theorem}\label{th-C} (Asymptotic behavior of $C_n$)\\
We have one the following two cases:
\begin{itemize}
\item If $\sum_{n=0}^{\infty} p_n <\infty$, we have $C_\infty=\sup_n C_n<\infty$ a.s.;
\item If $\sum_{n=0}^{\infty} p_n =\infty$, we have
$C_\infty=\sup_n C_n =\infty$ and  
$C_n/(\sum_{i=0}^{n-1} p_i) \almostsure 1$.  In particular, when  
$\sum_{n=1}^{\infty}(1 - p_n ) < \infty$, there exists a random time  $N$ such
that, for any time $n \ge  N$,  a new color is observed. 
Hence, each color will be observed a finite number of times, and only finitely many colors will end up having a number of observations larger than $1$.
\end{itemize}
\end{theorem}
{The equations governing the model dynamics in~\cite{Collevecchio} are exactly the same of those in~\eqref{model-equations},  
provided we set, for each $k$, the function  $f_k$ in~\cite{Collevecchio} equal to our update function~$F$ and replace $L_n$, $s_{n-1}$ and $A_k(n)$ in~\cite{Collevecchio} by 
    $C_n$, $p_{n-1}$ and $K_{n,c}$, respectively. In other words, each created group in~\cite{Collevecchio} corresponds to an observed color in the present urn model.  
     Hence, 
from Theorem 1.1 and Remark 1.2 in \cite{Collevecchio} (recall that our $F$ is strictly increasing by assumption), we obtain the following.}
\begin{theorem}
\label{Thm:Kcn}
If $p_n\le p<1$ for each $n\geq 1$, then we have the following two cases:
\begin{itemize}
	\item If $\sum_{n=1}^\infty 1/F(n) < \infty$, then  there will be, a.s., exactly one color observed an infinite number of times,  
	all the other observed colors being observed a finite  number of times.
	\item  If $\sum_{n=1}^{\infty} 1/F(n) = \infty$, then all the observed colors will be, a.s., observed infinitely often.  
	More formally, we have $K_{n,c}\uparrow \infty$ for each observed color $c$.
\end{itemize}
\end{theorem}

In the particular case when $F(x)=\rho x$ with $\rho>0$, under the condition that $p_n$ converges to $p\in [0,1)$ in a suitable way, we can obtain the rate at which the number of times a color is observed grows to $\infty$.
\begin{theorem}\label{th-K} Assume $F(x)=\rho x$ with $\rho>0$ and $p_n=p+O(1/n^\beta)$ with $0\le p<1$ and $\beta>p$. Then, for every observed color $c$, we have 
$$
\frac {K_{n,c}}{n^{1-p}}\almostsure K(c)\in (0,\infty)
\qquad\hbox{(i.e. $K_{n,c}\stackrel{a.s.}\sim K(c)\,n^{(1-p)}$)}\,.
$$
\end{theorem}
\begin{proof} Set $T_n=\sum_{c'\in \mathcal{C}_{n}} F(K_{n,c'})$, which is the total number of balls in the urn at time $n$. Moreover, for each observed color $c$, denote by $N(c)\geq c$ the first time at which we observe $c$ (note that $N(c)<\infty$ by definition for each observed color $c$) and set $P_n(c)=F(K_{n,c})/T_n$, the proportion of balls of color $c$ in the urn at time $n$. Then 
$P_{N(c)}(c)=F(1)/T_{N(c)}>0$. For $n\ge N(c)$, we get  
\begin{align*}
P_{n+1}(c)-P_n(c)
            &= \frac{F(K_{n+1,c})} {T_{n+1}} -P_n(c) \\
            &= \frac{F(K_{n+1,c})-F(K_{n,c})} {T_{n+1}} + \frac {F(K_{n,c})}{T_n}
                       \times \frac {T_n}{T_{n+1}} -	P_n(c) \\
            &= \frac{F(K_{n+1,c})-F(K_{n,c})} {T_{n+1}} - 
                       \frac {T_{n+1} -T_n}{T_{n+1}} P_n(c).         
	\end{align*}	
In this model, we have $T_n = \rho n$, and we obtain
$$
P_{n+1}(c)
            = \frac{\rho K_{n+1,c}-\rho K_{n,c}} {\rho(n+1)} + P_n(c) - P_n(c)
                       \frac {\rho}{\rho(n+1)}.$$
These dynamics amount to 
    \begin{align*}
		P_{n+1}(c)&=
		\Big(1-\frac{1}{n+1}\Big) P_n(c)+ \frac{1}{n+1} Y_{n+1}(c)
		,\\
Y_{n+1}(c)&= K_{n+1,c}- K_{n,c}.
\end{align*}	

Note that $\E[Y_{n+1}(c)\, |\, S_1,\dots, S_n]=(1-p_n)P_n(c)=(1-p)P_n(c)+\rho_n(c)$ 
with $0<1-p\le 1$ and 
$\rho_n(c)=(p-p_n)P_n(c)$, and so $\rho_n(c)/P_n(c)\almostsure 0$ and 
$\rho_n(c)=O(1/n^\beta)$ with  $\beta>p$. 
Applying Theorem~S1.3 in \cite{ale-cri-ghi} (with $N(c)$ as the initial time), we get 
$$n^{p} \, P_n(c)\almostsure 
{\widetilde P}(c)\in (0,\infty).
$$
To conclude, it is enough to observe that 
{
$$
n^{-(1-p)} K_{n,c}=n^{-(1-p)}\frac{T_n P_n(c)}{\rho}=n^{p} P_n(c)
\almostsure  \widetilde{P}(c),
$$ 
and so the thesis holds true with $K(c)=\widetilde{P}(c)$. }
\end{proof}
{
The next theorem deals with the case of an update function $F$ of the form $F(x)=\rho x +\widetilde{\rho}$, with $\widetilde{\rho}\neq 0$. 
\begin{theorem}
\label{thm:Behavior_Pn_Kn_rho_tilde_not_zero} 
Assume $F(x)=\rho x+\widetilde{\rho}$, with $\rho>0$, $\widetilde{\rho}\neq 0$ and $\rho+\widetilde{\rho}> 0$. Let 
$$\delta = \frac{1-p}{1+p\eta}=
    1-p\, \frac{1+\eta}{1+p\eta}\in (0,1],$$ 
    with $\eta=\widetilde{\rho}/\rho$. 
    Assume also  
$p_n=p+O(1/n^{\beta})$, with $0\leq p<1$ and $\beta>0$.  
Then, for every observed color $c$, we have 
	$$
    \frac{K_{n,c}}{n^{\delta}}\almostsure K(c)\in [0,\infty).
    $$	
    Consequently,  we have $K_{n,c}\stackrel{a.s.}=O(n^\delta)$. 
\end{theorem}
\begin{proof} With the same notation and computations of the proof of Theorem~\ref{th-K}, for $n\geq N(c)$, we have  \begin{align*}
P_{n+1}(c)-P_n(c)
            = \frac{F(K_{n+1,c})-F(K_{n,c})} {T_{n+1}} - 
                       \frac {T_{n+1} -T_n}{T_{n+1}} P_n(c).         
	\end{align*}	
Since $F(x)=\rho x +\widetilde{\rho}$, we get 
$T_{n}=\rho n + \widetilde{\rho}C_n$, which is greater than $\rho n$ when $\widetilde{\rho}>0$ and  greater than $(\rho+\widetilde{\rho}) n$ when $\widetilde{\rho}<0$ and, 
	by Theorem~\ref{th-C}, we have $T_n\stackrel{a.s.}\sim (\rho+ p\widetilde{\rho}) n$. 
	Moreover, we can write 
	\begin{equation*}
		\begin{split}
			P_{n+1}(c)-P_n(c)&=
			\frac{\rho(K_{n+1,c}-K_{n,c})}{T_{n+1}} -	
			P_n(c)\, \frac{\rho +\widetilde{\rho}\,\indicator_{\{S_{n+1} \not \in {\mathcal C}_n\}}}{T_{n+1}}\\
			&= \frac{\rho(K_{n+1,c}-K_{n,c})}{T_{n}} -	
			P_n(c)\, \frac{\rho +\widetilde{\rho}\,\indicator_{\{S_{n+1} \not \in {\mathcal C}_n\}}}{T_{n}} +\epsilon_{n+1}(c)\\
			&= \rho \frac{Z_{n+1}(c)}{T_n} + \epsilon_{n+1}(c)\,,
\end{split}
\end{equation*}
where 
$$
			Z_{n+1}(c)=(K_{n+1,c}-K_{n,c}) - P_n(c)(1 +\widetilde{\rho}/\rho\,\indicator_{\{S_{n+1} \not \in {\mathcal C}_n\}})
			$$
and 
$$
\epsilon_{n+1}(c)=\rho Z_{n+1}(c)\left(\frac{1}{T_{n+1}}-\frac{1}{T_n}\right)
$$
is such that $|\epsilon_{n+1}(c)|\leq r/n^2$ for a suitable constant $r$. 
Setting, 
\begin{align*}
\Delta M_{n+1}(c) &=Z_{n+1}(c)-\E[Z_{n+1}(c)\, | \, S_1,\dots, S_n]\\
                             &=	Z_{n+1}(c)-(1-p_n)P_n(c)+P_n(c)(1+\widetilde{\rho}/\rho \, p_n),
\end{align*}
      we obtain
\begin{equation*}
			P_{n+1}(c)-P_n(c)=
      - P_n(c) \frac{\rho}{T_n}p_n(1 + \widetilde{\rho}/\rho) + \frac{\rho}{T_n} \Delta M_{n+1}(c) + \epsilon_{n+1}(c)\,,
			\end{equation*}
            that is
            $$
            P_{n+1}(c)=\left(1-\rho u_n/T_n \right)P_n(c) + 
            \frac{\rho}{T_n} \Delta M_{n+1}(c) + \epsilon_{n+1}(c)\,,
            $$
			where $u_n=p_n(1 + \eta)\to p(1+\eta)$ and so 
            \begin{equation}\label{eq-as-conv-seq}
			n \, \frac{\rho u_n}{T_n} =\frac{\rho n}{T_n}u_n  \almostsure 
			p\frac{1 + \eta}{1 + p \eta}=1-\delta \,.
			\end{equation}
            Note that $0<\delta\leq 1$ (because $0\leq p<1$ and $\eta=\widetilde{\rho}/\rho>-1$ by assumption). 
Now, let us choose  $n^*$ large enough so that for $n\geq n^*$ we have $p_n<1$ and so 
$$
\rho u_n/T_n=p_n\frac{1+\eta}{n+\eta C_n}\leq
p_n<1
$$
and we can define 
		\begin{equation*}
	\zeta_{n^*}=1\qquad	\zeta_n=
    \prod_{j=n^*+1}^n \frac{1}{1- \rho u_{j-1}/T_{j-1}}
    \quad\hbox{for } n\geq n^*+1\,.
        \end{equation*}
Then, for $n\geq \max(n^*,N(c))$, setting 
		$P^*_n(c)=\zeta_n P_n(c)$, we obtain
        \begin{equation*}
        \begin{split}
P^*_{n+1}(c)&=(1-\rho u_n/T_n)\frac{\zeta_{n+1}}{\zeta_n}P^*_n(c)+
\frac{\rho}{T_n}\zeta_{n+1}\Delta M_{n+1}(c)+\zeta_{n+1}\epsilon_{n+1}(c)\\
&=P^*_n(c)+
\frac{\rho}{T_n}\zeta_{n+1}\Delta M_{n+1}(c)+\zeta_{n+1}\epsilon_{n+1}(c)\\
        \end{split}
        \end{equation*}
        and so, since $\zeta_{n+1}$ is $\sigma(S_1,\dots,S_n)$-measurable, we find   
        $$
\E[P^*_{n+1}(c)|S_1,\dots,S_n]=P^*_n(c)  + \zeta_{n+1} \E[\epsilon_{n+1}(c)|S_1,\dots,S_n]  
\leq P^*_n(c)+r \zeta_{n+1}/n^2.
        $$
        This means that the stochastic process $(P^*_n(c))$ is a nonnegative almost super-martingale and, in order to prove its almost sure convergence toward a nonnegative real random variable, it is enough to prove that 
        $\sum_n \zeta_{n+1}/n^2<\infty$ a.s. (see \cite{RobSie1971}). To this purpose, since 
        \eqref{eq-as-conv-seq} holds, 
        we apply Lemma~\ref{lemma-prod} to almost every realization of the random sequence  
        $\rho u_{n}/T_{n}$ and we get  
$\zeta_n\stackrel{a.s.}=O( n^{ (1-\delta)+\varepsilon})$ for every $\varepsilon >0$. As a consequence, we have 
$\sum_n\zeta_{n+1}/n^2\stackrel{a.s.}=\sum_n O(1/n^{1+\delta-\varepsilon})$ and,  since $\delta>0$, we can take $0<\varepsilon<\delta$ and obtain the a.s. convergence of the series. Hence, 
$P^*_n(c)$ almost surely converges to a nonnegative real random variable $P^*(c)$.  
Since the sequence $(\rho u_{n}/T_{n})$ has the same form as the one considered in Corollary~\ref{cor-prod} and since by Theorem~\ref{th-C} we have  
$C_n\stackrel{a.s.}=O(\sum_{i=0}^{n-1} p_i)$, we can apply the last part of Corollary~\ref{cor-prod}. We distinguish the case $p=0$ and the case $p>0$. 
In the first case,  from the Corollary~\ref{cor-prod}, we immediately get 
$\zeta_n\stackrel{a.s.}\sim L\, n^{1-\delta}$ where $L$ is a random variable taking values in 
$(0,\infty)$.  
In the case $p>0$, in order to obtain $\zeta_n\stackrel{a.s.}\sim L\, n^{1-\delta}$ with (random) $L\in (0,\infty)$ from Corollary~\ref{cor-prod},   
we need to prove that $C_n-\sum_{i=0}^{n-1} p_i\stackrel{a.s.}=O(n^\alpha)$ for a suitable $\alpha<1$. To this purpose, we observe that, 
since $C_0=0$, $C_n=\sum_{i=0}^{n-1}B_i$, with 
$\E[B_i|B_0,\dots, B_{i-1}]=p_{i}$, 
the process $M_0=0$, $M_n=C_n-\sum_{i=0}^{n-1}p_i
=\sum_{i=0}^{n-1}(B_i-p_i)$ 
is a martingale with increments bounded in absolute value by one.
Hence, by the Azuma--Hoeffding inequality, for every
$\alpha\in(\frac12,1)$,
\[
\Prob(|M_n|\ge n^\alpha)
\le
2\exp\left(-\frac{n^{2\alpha-1}}{2}\right).
\]
The right-hand side is summable with respect to $n$. Therefore,
the Borel--Cantelli lemma yields
\[
C_n-\sum_{i=0}^{n-1}p_i
=O(n^\alpha)
\qquad\text{a.s.}
\]
for every $\alpha\in(1/2,1)$. Therefore, we can apply Corollary~\ref{cor-prod} also in the case $p>0$. Hence, also in this case,
\(
\zeta_n\stackrel{a.s.}\sim L n^{1-\delta}\) with random \( L\in(0,\infty).
\)
Therefore, we have 
$$ n^{1-\delta} P_n(c)=\frac{n^{1-\delta}}{\zeta_n}P^*_n(c)\almostsure L^{-1}P^*(c)=
\widetilde{P}(c)\in [0,\infty)\,.
$$
To conclude, we observe that 
$$
n^{-\delta} K_{n,c}=n^{-\delta} \frac{T_{n}P_n(c)-\widetilde{\rho}}{\rho}=n^{1-\delta}P_n(c)(1 + \eta C_n/n)-n^{-\delta}\eta, 
$$
where $C_n/n\almostsure  p$. 
Hence, we obtain 
$$
n^{-\delta} K_{n,c}\almostsure K(c)=(1+p\eta)\widetilde{P}(c)=\frac{(1-p)}{\delta}\widetilde{P}(c)
\in [0,\infty)\,.
$$
\end{proof}
    \begin{remark}\rm In the case of the above theorem (where $\widetilde{\rho}\neq 0$), 
          to find conditions for $P(K(c)>0)=1$ (so that we can say $K_{n,c}\stackrel{a.s.}\sim K(c) n^\delta$)   is not simple as in the case $\widetilde{\rho}=0$ because, when $\widetilde{\rho}\neq 0$,  the dynamics of $P_n(c)$ are not exactly of the same type as the one considered in Theorem~S1.3 in \cite{ale-cri-ghi}.  Roughly speaking, using the approximation  
          $T_n/n\stackrel{a.s.}\sim (\rho+ p\widetilde{\rho})$, we have 
          (up to the approximation error arising from replacing $C_n/n$ with $p$), 
     \begin{equation*}
	\begin{split}
		P_{n+1}(c)&\approx
		\Big(1-\frac{1}{n+1}\Big) P_n(c)+ \frac{1}{n+1} \widehat{Y}_{n+1}(c)
		,\\
\widehat{Y}_{n+1}(c)&= \frac{K_{n+1,c}- K_{n,c} }{1+p\widetilde{\rho}/\rho}
- P_n(c)\, 
\frac{\widetilde{\rho}/\rho}{1+p\widetilde{\rho}/\rho}\indicator_{\{S_{n+1}\not \in {\mathcal C}_n\}}\\
&\qquad {} +P_n(c)\, \frac{p\widetilde{\rho}/\rho}{(1+p\widetilde{\rho}/\rho)}.
	\end{split}
\end{equation*}
Under the additional assumption $p_n-p=O(1/n^\beta)$ with $\beta>1-\delta$, the random variable $\widehat{Y}_{n+1}(c)$ has the conditional expectation of the form required in Theorem~S1.3 in \cite{ale-cri-ghi}, that is
\begin{equation*}
\begin{split}
\E[\widehat{Y}_{n+1}(c)\, | \, S_1,\dots, S_n]
&=\frac{1-p_n}{1+p\widetilde{\rho}/\rho}P_n(c)
-\frac{(p_n-p)\widetilde{\rho}/\rho}{1+p\widetilde{\rho}/\rho}P_n(c)\\
&= \delta P_n(c) + 
O\Big(\frac 1 {n^\beta}\Big). 
\end{split}
\end{equation*}
        However, $\widehat{Y}_{n+1}(c)$ does not take values in $[0,1]$. Nevertheless, supported by the simulation results, we conjecture that $P(K(c)>0)=1$ holds true also in the setting of Theorem~\ref{thm:Behavior_Pn_Kn_rho_tilde_not_zero} (possibly with some additional assumptions). Indeed, we can rewrite 
        the above dynamics as
\begin{equation*}
\begin{split}
		P_{n+1}(c)&\approx
		\Big(1-\frac{1+\alpha_n}{n+1}\Big) P_n(c)+ \frac{1}{n+1} {Y}_{n+1}(c) + P_n(c)\Delta L_{n+1}(c),\\
        Y_{n+1}(c)&=\frac{K_{n+1,c}- K_{n,c} }{1+p\widetilde{\rho}/\rho}\geq 0,
        \end{split}
\end{equation*}
where $(Y_{n+1}(c))$ is bounded with $\E[Y_{n+1}(c)|S_1,\dots,S_n]=\delta P_n(c) +O(1/n^\beta)$, $\alpha_n=O(1/n^\beta)$ and $(\Delta L_{n+1}(c))$ is 
a suitable bounded martingale difference. Hence, an extension of Theorem~S1.3 in \cite{ale-cri-ghi} to a dynamics of this type is needed and this is the object of a forthcoming research.
    \end{remark}
}
\subsection{Asymptotic approximation of the distribution of~$C_n$}
We can  get a meaningful result for the asymptotic distribution of 
the number~$C_n$ of observed colors for a general sequence~$\{p_n\}_{n=1}^\infty$,
provided the sequence follows some regularity conditions. \\
 
Recall the representation~(\ref{Eq:Cn})
of $C_n$ as a sum of independent  Bernoulli random variables.
The variables $B_i$ 
are uniformly bounded.  
Namely, we have
$$|B_i - p_i| \le  B_i +p_i \le 2,
           \qquad \mbox{for}\ 1\le i \le n.$$
The setup coincides with Example 7.15 of~\cite{Karr}, p.~193, where
Lindeberg's condition for normality is verified, under the regularity condition
$$\sum_{n=0}^\infty p_n(1-p_n) = \infty. $$ 
Under such a regularity condition, we obtain the central limit theorem
\begin{equation}
\frac {C_n - \sum_{i=0}^{n-1} p_i} {\sqrt{\sum_{i=0}^{n-1}p_i(1-p_i) }} \ \convD \ \normal (0,1). 
\label{Eq:CLT}
\end{equation}
The regularity condition is not too restrictive and is met in a wide variety of FAQ databases, such as for example $p_n$ is fixed to $p\in (0,1)$, or $p_n = 1/n$.

{We can find a Poisson approximation in total variation. } 
Toward this goal, we state a fundamental tool,
based on the total variation distance ($d_{TV}$) between  nonnegative probability
measures. For two such measures $\Prob$ and $\mathbb Q$ (both over $\mathbb N \cup \{0\}$),  the total variation distance
is defined as
{
\begin{equation*}
d_{TV}(\Prob,\mathbb Q)=  \sup_{A \subseteq \mathbb N \cup \{0\}} 
\big|\Prob(A) - {\mathbb Q}(A)\big|= 
\frac 1 2   \sum_{j=0}^\infty \big|\Prob(j) - {\mathbb Q}(j) \big|.
\end{equation*}   
}
Some authors use the notation $d_{TV} (X,Y)$ to mean the total variation distance between the law of the random variable~$X$ and that of $Y$. 
We use this shorthand for a briefer presentation.

The following result by Barbour and Hall~\cite{Barbour,Holst} is by now a solid classic. It is a tool toward
Poisson approximation and ultimately 
asymptotic normality.
\begin{theorem} 
\label{Thm:Holst}
Let $Y_0, \ldots, Y_{n-1}$ be independent
Bernoulli random variables, with $Y_i\ \eqlaw \ \Ber(p_i)$,
for $i=0, \ldots, n-1$. Define
$$\lambda_{n,1} = \sum_{i=0}^{n-1} \E[Y_i] = \sum_{i=0}^{n-1} p_i,
\qquad
\text{and}\qquad  \lambda_{n,2} =\sum_{i=0}^{n-1} p_i^2.
$$
Then, we have
$$d_{TV}\Big( \sum_{i=0}^{n-1}  Y_i,  \Poi (\lambda_{n,1})\Big)
       \le (1 - e^{-\lambda_{n,1}}) \, \frac            
    {\lambda_{n,2}} {\lambda_{n,1}}.
    $$
\end{theorem}
Under a regularity condition on the sequence $\{p_n \}_{n=1}^\infty$, forcing
$\lambda_{n,2}/\lambda_{n,1}$ to converge to 0, the Poisson law  
{ 
provides an asymptotic approximation in total variation.}

For example, in the FAQ context, with $p_0=1$ and $p_i = 1/i$ for $i\geq 1$ 
(i.e., $Y_i=B_i \stackrel{d}= \Ber (1/i)$), the scheme is under 
triggers of diminishing strength and early questions are  
more likely
in the future than the appearance of new questions. So, we have
$$\lambda_{n,1} = 1+ \sum_{i=1}^{n-1} \frac 1 i =
1 + H_{n-1} \sim \ln n,
\quad
\mbox{and}\quad  \lambda_{n,2} =1+ \sum_{i=1}^{n-1} \frac 1 {i^2} \sim 1+\frac {\pi^2} 6,$$
and $H_{n-1}$ is the harmonic number of order $n-1$.  

By the Barbour-Holst bound, we find that, uniformly in $x\in \mathbb R$, the probability
$\Prob(C_{10^7} \le x)$ does not differ from the probability 
$\Prob(\Poi(1+H_{10^7-1}) \le x) \approx \Prob(\Poi(17.695311)\le x)$  
by more than  
$(1+\pi^2/6)/ (1+H_{10^7-1}) \approx 0.1494$.

As the Barbour-Holst bound is uniform, we can assert that 
\begin{align*}
&\Big|\Prob\big( C_n \le 1+H_{n-1}+ x \sqrt {1+H_{n-1}}\,\big) \\
          &\qquad \qquad {} -
                    \mathbb P \big(\Poi (1+H_{n-1}) \le 1+H_{n-1} + x \sqrt {1+H_{n-1}}\, \big)\Big|\\
    &\qquad \le  
    d_{TV}\big(C_n,  \Poi (1+H_{n-1})\big)\\
   &\qquad \le \frac {1+\pi^2/6} {1+H_{n-1}}\\
   &\qquad\to 0.
\end{align*} 
In other words, we have
\begin{align*}
\Big|\Prob\Big( \frac {C_n - H_{n-1}-1} {\sqrt {1+H_{n-1}}} \le x\,\Big) -
                      \mathbb P \Big(\frac {\Poi (H_{n-1}+1) -  H_{n-1}-1}  { \sqrt {1+H_{n-1}}} \le x\Big)\Big |
                      \to 0.
\end{align*} 
{
Standardized Poisson distributions with increasing parameters converge to the standard normal distribution $\normal(0,1)$. That is to say }
$$ \frac {C_n - \ln n} {\sqrt {\ln n}}\ \convD \ \normal(0,1),$$
where asymptotic adjustments to replace the harmonic number with a logarithm are made
via Slutsky's theorem.
\subsection{The stationary color-frequency distribution}
Suppose the update function
$F$ is differentiable everywhere, except at a finite number of points. 
Recall that $T_n=\sum_{c\in {\cal C}_{n}} F(K_{n,c})$, which is the total number of balls 
in the urn at time~$n$. 

Set $Q_{n,k}=card\{c\in \mathcal{C}_n: K_{n,c}=k\}$, which  
represents the number of  colors in the urn that appear exactly $k$ times in the sequence $\mathcal S$ by time $n$, that is the number of colors $c$ such that $K_{n,c}=k$. 
We have $\sum_{k\ge 1} Q_{n,k}=C_n$. Assuming that the limits 
$q(k)=\lim_n Q_{n,k}/C_n$ exist and $\sum_k q(k)=1$, we call $k\mapsto q(k)$ the (stationary) color-frequency distribution.  In the following, 
we characterize it. 

We embed the discrete-time urn process in real time to enable
differential equation methods. The starting point is to replace discrete $n$ 
by continuous~$t$.  We write an equation for $Q_{t,k}$, with $k\geq 2$:
\begin{equation*}
  \begin{split}
  \frac{\partial Q_{t,k}}{\partial t}&=
  -\frac{Q_{t,k}(1-p_t) F(k)}{T_t}
  +\frac{Q_{t,k-1}(1-p_t) F(k-1)}{T_t}\\
  &=
  -(1-p_t) \frac{(Q_{t,k}-Q_{t,k-1})F(k)+Q_{t,k-1}\big(F(k)-F(k-1)\big)}{T_t}
  \\
  &\approx - \frac{(1-p_t)}{T_t}\Big(\frac{\partial \big(F(k)Q_{t,k}\big)}{\partial k}\Big).
  \end{split}
\end{equation*}
{The following is a plausible heuristic to develop the behavior of $q(k)$ and, in the next subsection, the frequency-rank curve. It has been tested 
before in~\cite{ale-cri, Tria2, Tria3}. } 
Using the asymptotic relations
$Q_{t,k}\approx q(k) C_t$ and $\frac{\partial Q_{t,k}}{\partial
  t}\approx q(k) \frac{dC_t}{dt} \approx q(k) p_t$, from the above relation, we get
  \begin{equation}\label{eq-freq}
q(k) p_t\approx -\frac{(1-p_t) C_t}{T_t}\Big(\frac{d\big(F(k)q(k)\big)}{dk}\Big).
  \end{equation}
We need for $\frac{p_t}{(1-p_t)C_t}T_t$ to converge almost surely to a strictly positive finite quantity. Assume
\begin{equation}\label{ass-limit-ell}
\frac{p_t}{(1-p_t)C_t}\, T_t\almostsure \ell \in 
(0,\infty),
\end{equation}
so that \eqref{eq-freq} becomes 
$$
q(k)\ell \approx -\Big(\frac{d\big(F(k)q(k)\big)}{dk}\Big).
$$
Since  $\frac{d(F(k)q(k))}{dk}=\frac{dF(k)}{dk}q(k)+F(k)\frac{dq(k)}{dk}$, we obtain
$$
q(k) \ell\approx -\Big(\frac{dF(k)}{dk}q(k)
         +F(k)\frac{dq(k)}{dk}\Big),
$$
leading to  
\begin{equation}\label{eq-freq-new}
 \frac{1}{q(k)}\Big(\frac{dq(k)}{dk} \Big)\approx 
 -\Big( \frac{\ell}{F(k)}+\frac{1}{F(k)}\Big(\frac{dF(k)}{dk}\Big)\Big).
\end{equation}
If there exists a function $H$ such that $\frac{dH}{dk}=\frac{1}{F}$, we obtain 
\begin{equation}\label{eq-ln-q}
\ln\big(q(k)\big)\approx -\ell H(k) -\ln\big(F(k)\big) + M,
\end{equation}
where $M$ denotes a constant. 
Note that, in order to have  $\sum_k q(k)=1$, we need 
$\sum_k \exp(-\ell H(k))/F(k)<\infty$.  
Some examples illustrate these concepts.
\begin{example}\label{ex-1} \rm 
    Take $F(x)=\rho x + \widetilde{\rho}$ with $\rho >0$ and 
    $\rho+\widetilde{\rho}>0$. 
Then, we have 
$T_n=\rho \sum_{c\in{\mathcal C_n}}K_{n,c}+\widetilde{\rho}\, C_n = \rho n+\widetilde{\rho}C_n$. Furthermore, if we take $p_n\to p\in (0,1)$, 
by Theorem~\ref{th-C}, we also have $C_n\stackrel{a.s.}\sim p n$, and we get 
\begin{equation*}
    \begin{split}
        \frac{p_t}{(1-p_t)C_t}\, T_t 
        &\stackrel{a.s.}\sim 
\frac{p}{(1-p)p t}\,(\rho t + \widetilde{\rho} p t )\\
&\almostsure 
\ell= \frac{\rho + p\widetilde{\rho} }{1-p}=
\frac{\rho}{1-p} \Big(1+ p \frac{\widetilde{\rho}}{\rho}\Big)> \rho>0,
    \end{split}
\end{equation*}
where the first inequality holds since $\widetilde{\rho}/\rho>-1$ and $0<p<1$. 
Equation \eqref{eq-freq-new} becomes 
$$
 \frac{1}{q(k)}\Big(\frac{dq(k)}{dk}\Big) \approx  
 -\Big(\frac \ell \rho+1\Big)
 \frac{1}{(k +\widetilde{\rho}/\rho)}. 
$$
So, we obtain  $q(k)\propto (k +\widetilde{\rho}/\rho)^{-(\ell/\rho +1)}$, that is, a (generalized) power-law with exponent strictly greater than $1$. 
In particular, in Simon's model, we have $\rho=1$ and $\widetilde{\rho}=0$, so that we find 
$\ell=1/(1-p)$ and $q(k)\propto k^{-[1/(1-p)+1]}$ (as reported in~\cite{Chung}). 
Note also that $\rho/\ell=\delta< 1$, where $\delta$ is the exponent  {appearing in~Theorem~\ref{thm:Behavior_Pn_Kn_rho_tilde_not_zero}.} 
(Here we have $\delta<1$ because $p>0$, while Theorem~\ref{thm:Behavior_Pn_Kn_rho_tilde_not_zero} also includes the case $p=0$, that leads to $\delta=1$.)
\end{example}

 \begin{example} \label{ex-2}\rm Set $F(x)=\rho x + \widetilde{\rho}$ with 
 $\rho >0$ and $\rho+\widetilde{\rho}>0$ and 
take
 $p_n\propto 1/ n^{1-\theta}$ with~$\theta\in (0,1)$.  
 By Theorem~\ref{th-C},
 we get 
 $p_n/C_n\stackrel{a.s.}\sim \theta/n$, $T_n\stackrel{a.s.}\sim \rho n$  and so 
 $$
 \frac{p_t}{(1-p_t)C_t}\, T_t \stackrel{a.s.}\sim
\theta t^{-1} \rho t   
 \almostsure \ell=\theta \rho\,.
 $$
 
 Hence, equation \eqref{eq-freq-new} becomes 
 $$
 \frac{1}{q(k)}\Big(\frac{dq(k)}{dk}\Big) \approx 
 -(\theta + 1) \frac{1}{ k + \widetilde{\rho}/\rho},
$$
and we obtain $q(k)\propto (k +\widetilde{\rho}/\rho)^{-(1 + \theta)}$, again a  (generalized) power-law with exponent strictly greater than $1$, as known for Simon's model (that is when $\rho=1,\widetilde{\rho}=0$) with $p_n\propto 1/ n^{1-\theta}$, where $\theta\in (0,1)$, see \cite{Zanette}. 
\end{example}

\begin{example} \label{ex-3}\rm Set $F(x)=\rho x + \widetilde{\rho}$ with $\rho >0$ and  $\rho+\widetilde{\rho}>0$ and 
take $p_n\propto 1 / n$. By Theorem~\ref{th-C}, we get
 $ p_n/C_n\stackrel{a.s.}\sim (n \ln n )^{-1}$, $T_n\stackrel{a.s.}\sim \rho n$ and so 
 $$
 \frac{p_t}{(1-p_t)C_t}\, T_t \stackrel{a.s.}\sim 
 \frac{1}{t \ln t} \rho t \almostsure \ell=0\,. 
$$
Since $\ell=0$, in this case assumption~\eqref{ass-limit-ell} is not valid.
\end{example}
\begin{example} \label{ex-4}\rm 
Take $F(x)=x^{1/\rho}$, with $\rho>1$. 
Recall the inequalities 
$$
\|v\|_\rho=\Big(\sum_{i=1}^N |v_i|^{\rho} \Big)^{1/\rho}\leq 
\|v\|_1=\sum_{i=1}^N |v_i|\leq 
N^{1-\frac{1}{\rho}} \|v\|_\rho =N^{1-\frac{1}{\rho}} \Big(\sum_{i=1}^N |v_i|^\rho \Big)^{1/\rho},
$$
for any $v\in {\mathbb R}^N$. 
Applying the second inequality to the vector $v_{n}$ with dimension $C_n$ and with (positive) components
 $(K_{n,1})^{1/\rho}, \ldots, (K_{n,C_n})^{1/\rho}$ and recalling that 
$\sum_{ c\in{\cal C}_n } K_{n,c}=n$, we have 
$$
T_n=\sum_{c\in {\mathcal C}_n}F(K_{n,c})=\sum_{c\in {\mathcal C}_n} (K_{n,c})^{1/\rho} \le 
C_n^{1-\frac{1}{\rho}}\Big(\sum_{c\in {\mathcal C}_n} K_{n,c} \Big)^{1/\rho}
={C_n^{1-\frac{1}{\rho}}} n^{1/\rho}.
$$
Therefore, we get 
$$
\frac{p_t}{(1-p_t)}\Big( \frac{T_t}{C_t}\Big) \le 
\frac{p_t}{(1-p_t)} \Big(\frac{C_t^{1-\frac{1}{\rho}}t^{1/\rho}}{C_t}\Big)=
\frac{p_t}{(1-p_t)} \Big(\frac{t}{C_t} \Big)^{1/\rho}\,.
$$
If $p_n$ is such that $(p_n/(1-p_n))(n/C_n)^{1/\rho}\to 0$,  we obtain $\ell=0$ and so assumption \eqref{ass-limit-ell} is not verified.  
For instance, if $p_n\propto 1/n^{1-\theta}$, with $\theta\in (0,1)$, 
then $(1-p_n)\almostsure 1$. 
By Theorem~\ref{th-C}, we have $C_n\propto n^\theta$, a.s.   
and so 
$$
\frac{p_t}{(1-p_t)} \Big(\frac{t}{C_t}\Big)^{1/\rho}
\stackrel{a.s.}{\propto} 
\frac{t^{1/\rho}}{ t^{1-\theta}\,t^{\theta/\rho}}
= 1/t^{(1-\theta)(1-\frac{1}{\rho})}\to 0.
$$ 
If $p_n\propto 1/n$,  then, by~Theorem~\ref{th-C},  we have $C_n\propto \ln(n)$ a.s., and so 
$$
\frac{p_t}{(1-p_t)}\Big(\frac{t}{C_t}\Big)^{1/\rho}
\stackrel{a.s.}\propto 
\frac{t^{1/\rho}}{t} \Big(\frac{1}{(\ln t)^{1/\rho} }\Big)
=\Big(\frac{1}{t}\Big)^{1-1/\rho}\Big(\frac{1}{\ln t }\Big)^{1/\rho}\to 0.
$$ 
\end{example}

\subsection{The frequency-rank curve}
Let us consider a generic sequence of elements (colors) and count the number of
occurrences of each element. Rank all
the elements according to their frequency of occurrence (rank $r=1$
corresponds to the most frequent element, the rank $r=2$ corresponds to
the second most frequent element and so on. 
The higher the rank, the
less frequent the element) 

Let us plot the number 
of occurrences versus the rank:  
This curve is  known as 
the {\em frequency-rank curve.} Our aim is to relate this curve to the stationary color-frequency distribution and to the update function~$F$ in the model.

Define $R_n(z)$ as the number of different colors with frequency $k\ge z$ at time $n$. 
That is, $R_n(z)$  is the number of elements in the 
sequence $\cal S$, such that $k \ge z$ at time $n$, i.e.~$R_n(z)=\sum_{k\geq z} Q_{n,k}$. 
Passing to~real time then to the asymptotics, since $Q_{t,k}\approx q(k)C_t$, we can write 
\begin{equation}\label{eq-R}
R(z)\propto \int_{z}^{\infty} q(k)\, dk,\qquad \mbox{i.e. } \frac{dR(z)}{dz}\propto - q(z).
\end{equation}
For instance, if $q(k)\propto (k+\eta)^{-\beta}$ with 
$\beta>1$ and $\eta>-1$, as in Examples~\ref{ex-1} and~\ref{ex-2} where $\beta=(\ell/\rho+1)$ and $\eta=\widetilde{\rho}/\rho$, we find
$$
R(z)\propto (z+\eta)^{-(\beta-1)}.
$$
Inverting this relation in order to obtain the behavior of the frequency-rank function $r\mapsto z(r)$, we obtain the  
(generalized) Zipf's law: 
\begin{equation}\label{eq-zipf}
\big(z(r)+\eta\big) \propto r^{-\alpha}, \qquad\mbox{with } \alpha= \frac 1 {\beta-1} \in (0,\infty).
\end{equation}
In particular, in Example~\ref{ex-1} we have $\alpha=\rho/\ell<1$ and it depends on $p$ and $\eta=\widetilde{\rho}/\rho$; while in  
Example~\ref{ex-2}, since $\ell/\rho=\theta$, we find the relationship $\theta\alpha=1$ between 
Heaps exponent~$\theta$ and (generalized) Zipf's exponent~$\alpha$, and so $\alpha>1$ and it does not depend on $\eta=\widetilde{\rho}/\rho$.
\\
\indent Moreover, from \eqref{eq-ln-q} we get 
\begin{equation}\label{eq:master2}
	\ln\big(q(k)\big) \approx - \ln\big(F(k)\big)  - \ell \int_{k_0}^{k} \frac{1}{F(x)}\, dx +M,
\end{equation}
where $M$ denotes a suitable constant. Suppose we observe dependency in the frequency-rank plot of the form
\begin{equation}\label{eq:Zipf-emp}
	g\big(z(r)\big) = -\widehat{a} \ln(r) + \widehat{b}, 
\end{equation}
with a strictly increasing function $g$ and estimated values $\widehat{a}>0$, $\widehat{b}$. Then, we
get
\begin{equation*}
	r =  \exp \Big( -\frac{g\big(z(r)\big)-\widehat{b}}{\widehat{a}} \Big)\,.
\end{equation*}
If $g$ is differentiable, we have 
\begin{equation*}
	\frac{d z(r)}{d r} =
	\frac{d g^{-1} \big(-\widehat{a} \ln(r) + \widehat{b}\big) }{d r} =
	-\frac{\widehat{a}}{g'(g^{-1}\big(\widehat{b}-\widehat{a}\ln r)\big)r}.
\end{equation*}
Since $z(r)$ is the inverse function of $R(z)$, from \eqref{eq-R}, 
we find
$$
\frac{d z(r)}{dr} \propto -\frac{1}{q\big(z(r)\big)}
$$
and we get
\begin{equation*}
\begin{split}
	q\big(z(r)\big) \propto -\frac{1}{\frac{d z(r)}{dr}}&=
	g'(g^{-1}\big(-\widehat{a}\ln(r)+\widehat{b})\big)\frac{r}{\widehat{a}}\\
    &=
	\frac{1}{\widehat{a}}g'(z(r))\exp \Big( -\frac{g\big(z(r)\big)-\widehat{b}}{\widehat{a}} \Big).
    \end{split}
\end{equation*}
Setting $k = z(r)$ in the above relation and taking the logarithm, we
find
\begin{equation*}
	\begin{split}
		\ln\big(q(k)\big) & \approx \ln\big(g'(k)\big) - \frac{g(k)}{\widehat{a}} + M_1
		\\
		& = - \ln\Big(\frac{\widehat{a}\ell}{g'(k)}\Big) - \ell \int_{k_0}^{k}
		\frac{1}{\widehat{a}\ell} g'(s) \,ds + M_2\,,
	\end{split}
\end{equation*}
where $M_1$ and $M_2$ denote suitable constants. 
{A term-by-term heuristic comparison with \eqref{eq:master2} suggests the relation}  
\begin{equation*}
	F(x)\approx \frac{\widehat{a}\ell}{g'(x)}\propto \frac{1}{g'(x)}.
\end{equation*}
Indeed, in Examples~\ref{ex-1} and~\ref{ex-2} we have the update function $F$ of the form $F(x)=\rho x +\widetilde{\rho}\propto (x+\widetilde{\rho}/\rho)$ and a (generalized)  Zipf's law with exponent $\alpha=\rho/\ell$ and $\eta=\widetilde{\rho}/\rho$ (see~\eqref{eq-zipf}),  which corresponds  to $g(x)=\ln(x+\eta)$ and $\widehat{a}$ equal to an estimate of $\alpha$ (see~\eqref{eq:Zipf-emp}). In particular, in the framework of Example~\ref{ex-1}, we can estimate $\eta=\widetilde{\rho}/\rho$ by 
\begin{equation}\label{eq:ratio_rhotilderho_estimator}
\widehat{(\widetilde{\rho}/\rho)}=\frac{1}{\widehat{p}}\Big(\frac{1-\widehat{p}}{\widehat{\delta}}-1\Big), 
\end{equation}
{provided $0<\widehat{p}<1$ and a power-law growth of $K_{n,c}$ with estimated exponent $0<\widehat{\delta}<1$.  }

\indent The appendix is devoted to some simulations that corroborate the obtained theoretical results.

\section{Simon's FAQ urn}
\label{Sec:Simon}
We analyzed the FAQ triggered urn asymptotically, as $n\to\infty$. 
Certainly, the practitioner is also interested in knowing the behavior of her database for small and moderate 
$n$, too. Exact expressions facilitate this assessment.
The classic Simon model~\cite{Simon} 
in which the probability of the $n${th} trigger $p_n =p\in (0,1)$ is fixed, and $F(x) =x$ is amenable to combinatorial analysis.
We can address in the exact sense a few FAQ issues of interest. \\

Recall that $F(K_{n,c})$ is the number of balls of color $c$ in the urn at time $n$. In the Simon FAQ model, one ball is added at a time, and so 
$K_{n,c} =F(K_{n,c})$ is the number of balls of color $c$ at time $n$.

\subsection{The number of colors}
\label{Sec:colors}
Recall~(\ref{Eq:Cn}).
It is straightforward 
to find 
$$\E[C_n] = 1 + p(n-1), \qquad \V[C_n] = p(1-p)(n-1),$$
with the exact 
\begin{equation}
\Prob(C_n =i) = {n-1\choose i-1} p^{i-1} (1-p)^{n-i} , \qquad \mbox{for \ }n \ge  i \ge 1,
\label{Eq:exact}
\end{equation}  
and, of course, $C_0 = 0$.
\begin{lemma}
\label{Lem:Xn0}
The probability that {color $c$ has not yet appeared} {by time~$n\ge c\geq 2$} is 
$$\Prob(K_{n,c} = 0) =  {n-2 \choose c-2} p^{c-2}  (1-p)^{n-c+1}+  \sum_{r=1}^{c-2} {n-2\choose r-1} p^{r-1}(1-p)^{n-r-1} .$$
\end{lemma}
\begin{proof}
For $K_{n,c}$ to be 0 at time $n$, 
it must have not appeared at time $n-1$. If $C_{n-1} < c$,  the event $K_{n,c} = 0$ occurs,
no matter what. However, if $C_{n-1} = c-1$, color $c-1$ is in the urn, for $K_{n,c}$ to stay at 0, the next Bernoulli random variable, $B_{n-1}$, must fail (no triggering). We have 
\begin{align*} 
\Prob(K_{n,c} = 0) 
         &= \sum_{r=1}^{c-1} \Prob(K_{n,c} = 0 \given C_{n-1}=r)\, 
              \Prob(C_{n-1}=r)\\
         &= \Big(\sum_{r=1}^{c-2} \Prob(K_{n,c} = 0 \given C_{n-1}=r)\, 
              \Prob(C_{n-1}=r)\Big) \\
          &\qquad\qquad {}+ \Prob(K_{n,c} = 0 \given C_{n-1}=c-1)\, 
              \Prob(C_{n-1}=c-1)\\
               &= \Big(\sum_{r=1}^{c-2}\Prob(C_{n-1}=r)\Big) + (1-p)\, \Prob(C_{n-1}=c-1)  .
\end{align*} 
The probabilities in this expression are given in~(\ref{Eq:exact}). 
\end{proof}
\subsection{Recursion of the average number of balls of a specific color}
\label{Sec:freq}
Recall that in Simon's urn $K_{n,c}$ is the number of balls of color~$c$ at time $n$. 
It is clear that color~$c$~does not appear before time $c$. So, we have
$K_{n,c} = 0$, for $n < c$ and color~$c$ cannot be born in the urn at time~$n$, unless balls of color $c-1$ 
exist in the urn (i.e., unless $K_{n-1,c-1} > 0$).   

At time $c$, a ball of color $c$ may or may not appear. 
It appears
only if the history sequence $\cal S$ begins with $1,2\ldots c-1$ (after $c-1$ successful triggers). 
Recall that the first trigger is always successful $(p_0=1)$. 
So, $K_{c,c}$ is a Bernoulli
random variable with success probability $1 \times p^{c-2}\times p=p^{c-1}$. Hence,  we have $\E[K_{c,c}] = p^{c-1}$,
an expression needed as a boundary condition for the recursion of the average. 
The solution is in terms of the gamma function~$\Gamma(.)$.
\begin{lemma}
\label{Lem:ave}
In the Simon urn,
for $n\ge c\ge 1$, the average number of times question $c$ is asked is
\begin{align*}
\E[K_{n,c}] =
      & \frac{\Gamma(n-p+1)}{\Gamma(n)}\Big(\frac {p^{c-1}}{(1-p)^c} \sum_{i=c+1}^n {{i-2} 
      \choose c-2} \,
      \frac{(1-p)^ i\, \Gamma(i)}{\Gamma(i-p+1)}\\
      &\qquad \qquad {}+
                      \frac{p^{c-1}\,\Gamma(c)}{\Gamma(c-p+1)}\Big).
\end{align*}  
\end{lemma}

\begin{proof}
We compute $\E[K_{n,c}]$ starting from a stochastic recurrence:
$$K_{n,c} = K_{n-1,c} +\indicator_{n,c},$$
where $\indicator_{n,c}$ is an indicator of the event of an increase in color $c$ by one ball at time $n$. 
{Let $\field_{n-1}$ be the sigma field
generated by the first $n-1$ insertions in the history sequence $\mathcal S$, 
that is $S_1,\dots,S_{n-1}$. 
 Now, let us  condition the stochastic recurrence  on $\field_{n-1}$ and $B_{n-1}$, that is on the sigma field
generated by $S_1,\dots,S_{n-1}$ and $B_{n-1}$, and obtain
$$\E\big[K_{n,c} \given \field_{n-1}, B_{n-1} \big] = K_{n-1,c} +
        \E\big[\indicator_{n,c} \given \field_{n-1} , B_{n-1}\big].$$ 
}
According to the construction algorithm of Simon's urn, we have cases. 
When there is no trigger, we choose from the existing balls, proportionately according
to the existing masses. Alternatively, when there is a trigger color~$c$ appears, but only if
there are balls of colors $1,2, \ldots,c-1$ in the urn (otherwise, the trigger is introducing a color 
numbered lower than $c$). Expressed conditionally, this argument gives      
$$\E\big[\indicator_{n,c} \given \field_{n-1} , { B_{n-1}}\big] = 
\begin{cases}
            \frac {K_{n-1,c}}{n-1}, & \mbox {if\ } B_{n-1}=0\ \mbox{and}\ K_{n-1,c} >0;\\
            1, &  \mbox {if\ } B_{n-1}=1\ \mbox{and}\ C_{n-1} = c-1;\\
            0, &\mbox {otherwise}.
\end{cases}$$
Now, take a double expectation to obtain an unconditional recursion: 
  $$\E\big[K_{n,c}] = \Big(1 
           + \frac {1-p}{n-1}\Big) \E\big[K_{n-1,c}] + p\,\Prob(C_{n-1} = c-1).$$
With $c$ held fixed,  this linear recursion is  in the standard form
$$a_n = g_n a_{n-1} + h_n,$$
with solution\footnote{The product is interpreted as 1, when it does not exist. In particular, at $j=n$, 
the summand is $h_n$.}
$$a_n = \sum_{i=c+1}^n h_i\prod_{j=i+1}^n g_j + a_c \prod_{j=c+1}^n g_j.$$
Specialized to the Simon FAQ urn, 
with 
$$a_n = \E[K_{n,c}], \qquad g _n = 1 + \frac {1-p}{n-1},$$
and   
$$h_n = p\,\Prob(C_{n-1} = c-1), 
     \qquad a_{c} =\E[K_{c,c}] = 
    p^{c-1},$$
we obtain the solution
$$\E[K_{n,c}]  = p \sum_{i=c+1}^n \Prob(C_{i-1} = c-1) \prod_{j=i+1}^n  \Big(\frac {j-p}{j-1}\Big)
          + 
           p^{c-1}\prod_{j=c+1}^n \Big(\frac {j-p}{j-1}\Big).$$
The probabilities for $C_{i-1}$ are given 
in~(\ref{Eq:exact}) and 
the products can be completed to gamma functions as in the statement of the lemma. 
Note that, for the first question, $c=1$, the probabilities $\Prob(C_{i-1} = c-1) = \Prob(C_{i-1} = 0)$ are
all $0$, as color $1$ exists for all $n\ge 1$. This is also consistent with the standard interpretation ${i-2\choose -1} =0$, for any $i-2\ge 0$.
\end{proof}
For color $c=1$, we can remove the sum from the solution,
greatly simplifying it. An asymptotic equivalent can be obtained via Stirling approximation 
of the gamma function. 
\begin{cor}
\label{Cor:Xn1}
In the Simon FAQ model, the average
number of times the first question is asked is {given} by
\begin{align*}
\E[K_{n,1}] &=\frac{\Gamma(n+1-p)} {\Gamma(2-p)\, \Gamma(n)} =
\frac{n^{1-p}} {\Gamma(2-p)} + O\Big(\frac1 {n^p}\Big), \qquad \mbox {as\ } n\to\infty.
\end{align*}
\end{cor}
For $c > 1$,
it is harder to get a closed-form expression for $\E[K_{n,c}]$. However, we can get asymptotics
with rates of convergence via Stirling approximation.
\begin{cor}
\label{Cor:Xnge1}
For $c \ge 2$, 
as $n \to \infty$, we have
$$
\E[K_{n,c}] \sim p^{c-1}\Big(\frac 1 {(1-p)^c}  \Lambda_\infty(c,p)+
                      \frac {\Gamma(c)}{\Gamma(c-p+1)}\Big) n^{1-p},
$$
where $\Lambda_\infty(c,p)$ is the convergent series
$ \sum_{i=c+1}^\infty {{i-2} 
      \choose c-2} \,
      \frac{ (1-p)^i\, \Gamma(i)}{\Gamma(i+1-p)}$.
\end{cor}
\begin{proof}
We demonstrate that the series
$$\Lambda_n(c,p) := \sum_{i=c+1}^n {{i-2} 
      \choose c-2} \,
      \frac{ (1-p)^i\, \Gamma(i)}{\Gamma(i+1-p)}$$ 
converges, say to $\Lambda_\infty(c,p)$.
By Stirling approximation, we have
$$ \frac{\Gamma(i)}{\Gamma(i+1-p)} \sim i^{p-1}, \qquad \mbox{as\ }i \to \infty,$$
and
$${i-2 \choose c-2} \sim \frac 1{(c-2)!}\, i^{c-2},\qquad  \mbox{as\ }i \to \infty.$$
The exponential decay {in $(1-p)^i$ dominates the polynomial factors}. By the ratio test, $\Lambda_n(c,p)$ converge, as ${n}\to\infty$. 

We can now assess the asymptotics of the exact expectation in Lemma~\ref{Lem:ave}.
Using Stirling approximation again, we get
$$
\E[K_{n,c}] =
       n^{1-p}\Big(1 + O\Big(\frac 1 n\Big) \Big)\Big(\frac 1{(1-p)^c} \Lambda_\infty (c,p)+
                      \frac {\Gamma(c)}{\Gamma(c-p+1)}\Big)p^{c-1}.
$$
\end{proof}
Corollaries~\ref{Cor:Xn1}--\ref{Cor:Xnge1} 
are consistent with Theorem~\ref{th-K}, which states that, for each color $c\geq 1$ (recall that we have $C_\infty=\infty$ when $p_n=p\neq 0$), we have $K_{n,c} 
\stackrel{a.s.}\sim K(c) n^{(1-p)}$, where $K(c)$ is  a suitable random variable with values in $(0,\infty)$. The exact {expectation} leads us to {conjecture} that the expectation 
of $K(c)$ is $(\Lambda_\infty(c,p) /(1-p)^c+\Gamma(c)/\Gamma(c-p+1))p^{c-1})$.  
{In order to formally prove that, we need to show that the above almost sure convergence is also convergence in mean.}
\subsection{Dynamic variations on Simon's urn}
\label{Sec:dynamic}
In Simon's model, we considered the case $p_n = p$, a constant. For a 
time-dependent sequence $p_n$,
we can still say something about the frequencies of asking certain questions. For instance,
the same logic we used to derive the average of $K_{n,1}$ in the case of constant
$p_n$, for a more general sequence leads us to 
$$ \E[K_{n,1}] =  \frac {n-{p_{n-1}}} {n-1}\,  \E[K_{n-1,1}],$$
with the initial condition $\E[K_{1,1}] =1$,
giving the solution
$$\E[K_{n,1}]= \prod_{i=2}^n \frac {i-{{p_{i-1}}}} {i-1} .$$    
As $n\to \infty$,  an increasing sequence $p_n$ like $1 - 1/n\to 1$ gives
$$\E[K_{n,1}] \to  {3.67608\ldots}\ .$$ 
In contrast,
a decreasing sequence $p_n$ like $1/n \to 0$ gives
$$\frac 1 n\, \E[K_{n,1}] \to {0.296675 \ldots}\ .$$ 

\section{Applications to real data}
We conducted an analysis on a real dataset from Amazon and corroborated the theory with multiple simulations.
\label{Sec:conc}
\subsection{Data description and preprocessing}
The analysis on real data is based on the Amazon Question and Answer (Q\&A) corpus, 
described in~\cite{wan-mcauley}. The datasets are publicly available and can be freely downloaded
from the project website: 

\url{https://cseweb.ucsd.edu/~jmcauley/datasets/amazon/qa/}.

This corpus contains user-generated questions and answers posted on Amazon product pages. 
We have selected three categories—\textit{Automotive}, \textit{Cell Phones and Accessories}, and \textit{Electronics}— that differ in theme, 
focus and user interaction patterns, and hence enable a comparative assessment of the proposed statistical methodology across distinct domains of consumer activity. 

The raw data required extensive preprocessing to ensure the syntactic validity of the textual content of each question, resulting in validated datasets comprising $n=88749$, $n=85053$, and $n=308932$ questions for the \textit{Automotive}, \textit{Cell Phones and Accessories}, and \textit{Electronics} categories, respectively.  
Since the question timestamp was not available, all validated entries were ordered chronologically according to the answer timestamp.

To analyze the semantic structure of the questions, after preprocessing, sparse features in the textual field \texttt{question} were filtered, retaining only the most informative terms for each category. 
As a result, each record (question) was represented by a set of numerical weights expressing the relevance of each informative term within that record relative to the rest of the dataset.
Finally, to cluster together semantically similar questions—corresponding to balls of the same color in the urn representation—an unsupervised $\kappa$-means algorithm was applied to the preprocessed datasets, with \(\kappa = 10000\), \(\kappa = 20000\), and \(\kappa = 40000\) clusters for the \textit{Automotive}, \textit{Cell Phones and Accessories}, and \textit{Electronics} datasets, respectively.
\subsection{Results}
The aim of this section is to illustrate that, for each of the three datasets considered, the empirical behavior is {consistent} with the triggered urn model introduced in this article. We show that the observed patterns, fitted curves, and estimated parameters {are broadly compatible with} those predicted by a triggering urn with $p_n \propto n^{-(1-\theta)}$ and an update function $F$ satisfying
$$
F(x)= \rho x +\widetilde{\rho}\propto x+\eta, 
$$
where the parameters $\theta\in (0,1)$ and $\eta=\widetilde{\rho}/\rho>-1$ 
may vary across applications.

{Specifically, building upon the previous theoretical results and the simulations performed, we consider the following  asymptotic (that is for large $n$) behaviors: }
\begin{itemize}
    \item[(i)] (\emph{Number of colors over time}) $\log_{10}(C_n)$ increases linearly with $\log_{10}(n)$, with slope equal to $\theta$;
    \item[(ii)] (\emph{Number of draws of a given color over time}) $\log_{10}(K_n(c))$ increases linearly with $\log_{10}(n)$, with slope equal to 1;
    \item[(iii)] (\emph{Frequency--rank curve}) $\log_{10}(z(r))$ decreases linearly with $\log_{10}(r)$, with slope equal to $-\alpha=-\theta^{-1}$.
\end{itemize}
Figures~\ref{fig:real_Automotive}, \ref{fig:real_CellPhones}, and \ref{fig:real_Electronics} show that the empirical patterns observed across all three datasets {are broadly consistent with} these linear log--log relations. 
{The corresponding $R^2$ values are reported in the captions of the three figures.}
In particular, clear linear trends emerge for: the growth of the number of colors (top panels),  the number of draws per color over time (bottom-left panels) and the frequency-rank curves (bottom-right panels).
{To assess the consistency between the theoretical and empirical relationships}, Table~\ref{Table real data} reports a detailed comparison between the theoretical slopes of these linear relationships (expressed as functions of $\theta$) and the empirical slopes obtained via linear regression for each of the three datasets.
Since the true value of the parameter $\theta$ is not available in real data analysis, we use the slope estimated from regression (i) on $C_n$ as an estimate $\widehat{\theta}$ of $\theta$.

\begin{figure}[!htbp]
    \centering
    \includegraphics[scale=0.3]{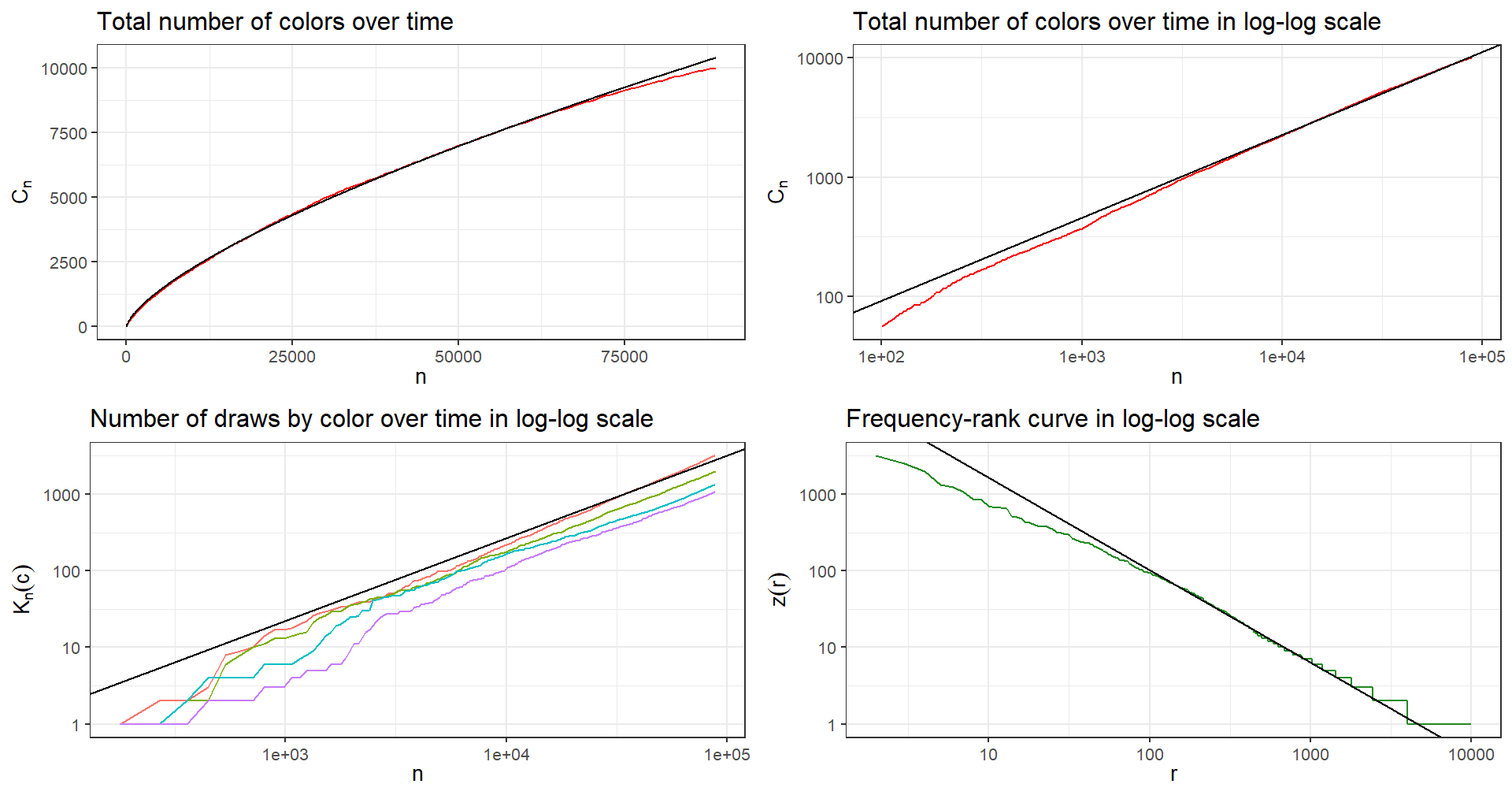}
    \caption{\textbf{Automotive dataset}: overview of the results. {The $R^2$ values of the linear regressions (i)-(ii)-(iii) are, respectively, $0.999$, $0.993$ and $0.997$.}}
    \label{fig:real_Automotive}
\end{figure}
\begin{figure}[!htbp]
    \centering
    \includegraphics[scale=0.3]{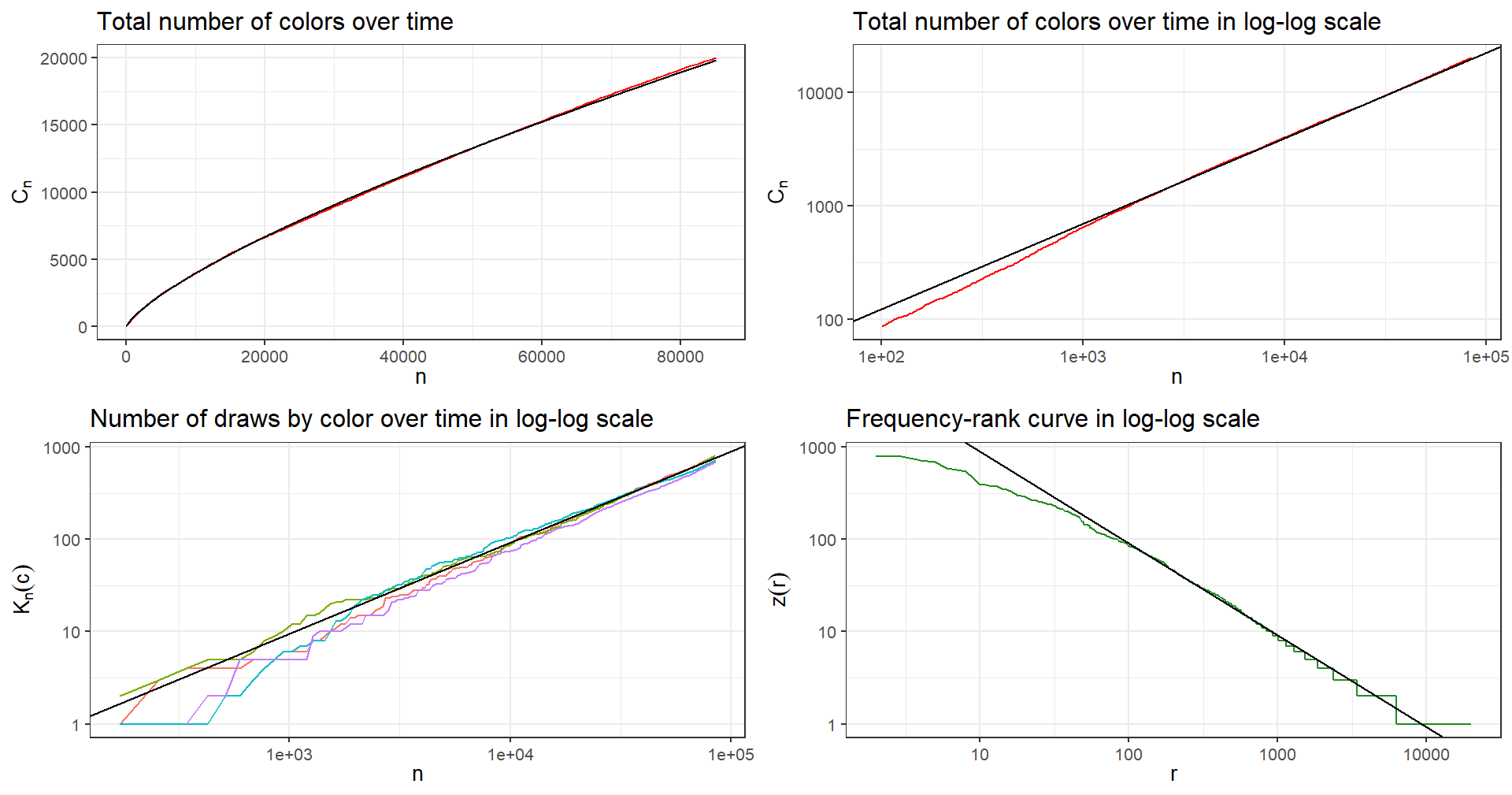}
    \caption{\textbf{Cell Phones dataset}: overview of the results. {The $R^2$ values of the linear regressions (i)-(ii)-(iii) are, respectively, $0.999$, $0.994$ and $0.997$.}}
    \label{fig:real_CellPhones}
\end{figure}
\begin{figure}[!htbp]
    \centering
    \includegraphics[scale=0.3]{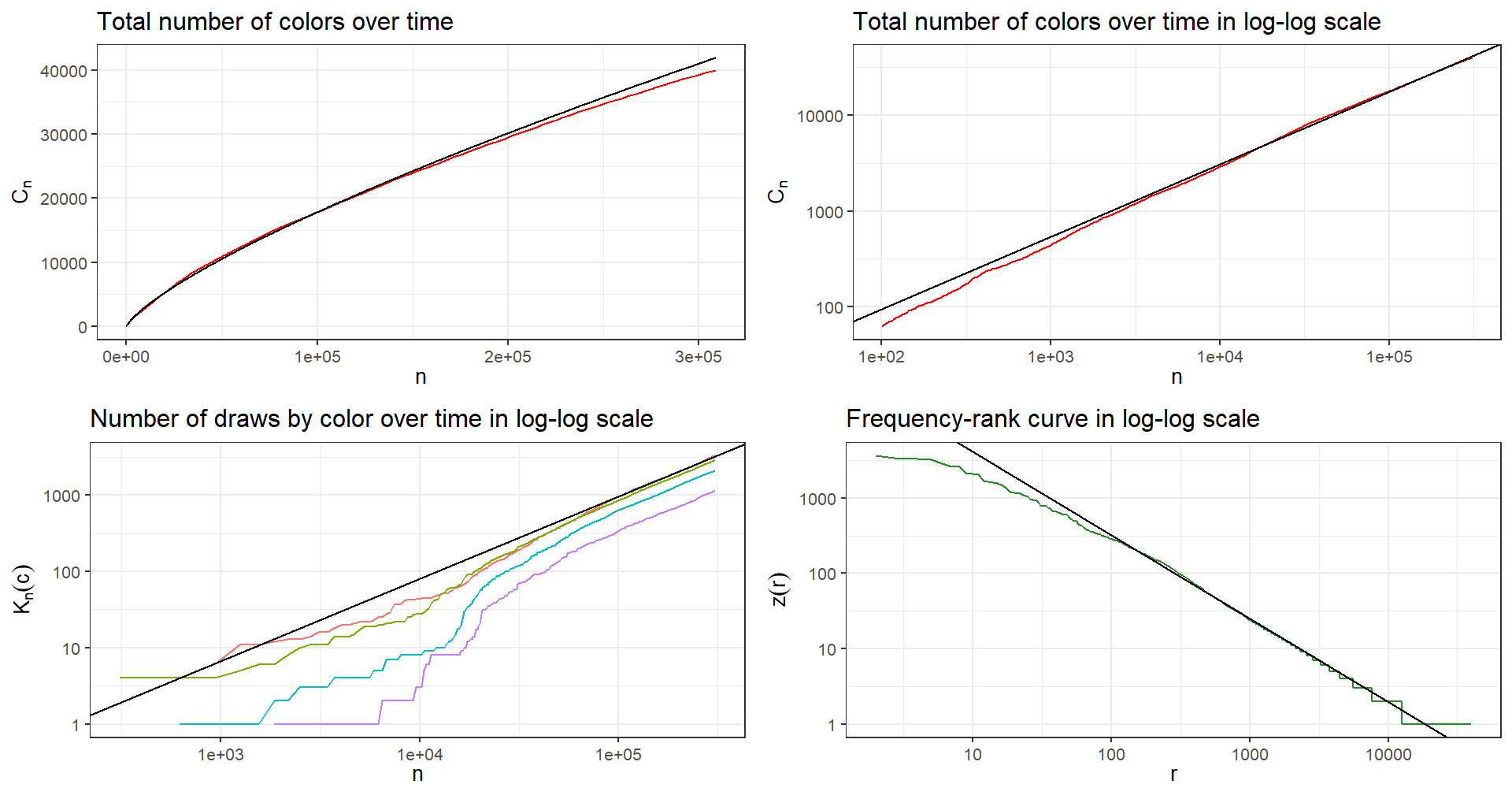}
    \caption{\textbf{Electronics dataset}: overview of the results. {The $R^2$ values of the linear regressions (i)-(ii)-(iii) are, respectively, $0.998$, $0.999$ and $0.995$.}}
    \label{fig:real_Electronics}
\end{figure}
\renewcommand{\arraystretch}{1.5}
\begin{table}[htbp]
\caption{Comparison of the empirical and theoretical slopes for the linear regressions (i)-(ii)-(iii) fitted to the three datasets.}
\label{Table real data}
\centering
\small
\begin{tabular}{|l|l|ccc|}
\hline
Regression & slope & Automotive & Cell Phones & Electronics \\
\hline
$\log_{10}(C_n)$ vs.\ $\log_{10}(n)$
    & emp ($=\widehat{\theta}$)       & $0.70$ & $0.75$ & $0.76$ \\
\hline
\multirow{2}{*}{$\log_{10}(K_n(c))$ vs.\ $\log_{10}(n)$} 
    & emp      & $1.08$ & $0.99$ & $1.08$ \\
    & theo     ($=1$) & $1.00$ & $1.00$ & $1.00$ \\
\hline
\multirow{2}{*}{\shortstack[c]{$\log_{10}(z(r))$ vs.\ $\log_{10}(r)$}} 
    & emp    & $-1.21$ & $-1.00$ & $-1.11$ \\
    & theo ($=-\widehat{\theta}^{-1}$) & $-1.43$ & $-1.33$ & $-1.32$ \\
\hline
\end{tabular}
\end{table}

\noindent{\bf Author contributions}\\
All the authors contributed equally to the present work.

\bigskip
\noindent{\bf Funding}\\
This research was supported by the “Resilienza Economica e Digitale” (RED)
project (CUP D67G23000060001) funded by the Italian Ministry of University
and Research (MUR) as “Department of Excellence” (Dipartimenti di Eccellenza
2023--2027, Ministerial Decree no. 230/2022).
\section*{Acknowledgments}
Irene Crimaldi is a member of the Italian Group “Gruppo Nazionale
per l’Analisi Matematica, la Probabilità e le loro Applicazioni” of the Italian Institute “Istituto Nazionale di Alta Matematica.”

Hosam Mahmoud completed this work while on sabbatical leave at The Catholic University of America  in Fall 2025. He thanks the Department of Mathematics and Statistics there for their gracious reception and hospitality and for an environment conducive to research under the leadership of Kiran Bhutani.

{The authors are indebted to two anonymous reviewers, whose careful scrutiny identified 
many places for improvement.} 

\appendix
{
\section{Technical results}
\begin{lemma}\label{lemma-prod}
Let $(x_n)_{n}$ be a sequence of nonnegative real numbers such that $n x_n \to x$ for some $x \geq 0$. Then, taking $n^*\geq 1$ such that  
$x_n<1$ for every $n\geq n^*$ and setting 
\[
\pi_n=\prod_{i=n^*}^{n}\frac{1}{1-x_i},
\]
we have 
\[
\pi_n=n^{x+o(1)}
\]
and so, for every $\varepsilon>0$,
$
\pi_n=O\left(n^{x+\varepsilon}\right).
$
\\
If, in addition,  
the series
$
\sum_{n\geq 1}\left(x_n-\frac{x}{n}\right)
$
is convergent, then there exists a constant $L\in(0,\infty)$ such that
\[
\pi_n\sim L n^x.
\]
\end{lemma}
\begin{proof} 
Using the Taylor expansion
\begin{equation}\label{eq-Taylor}
-\ln(1-x_n)=x_n+O(x_n^2),
\end{equation}
and the fact that $\sum_n x_n^2<\infty$, we obtain
\begin{equation*}\label{eq:lem-1}
\ln \pi_n
=-\sum_{i=n^*}^{n}\ln(1-x_i)=\sum_{i=n^*}^{n}x_i
  +O\Big(\sum_{i=n^*}^{n}x_i^2\Big)
=\sum_{i=n^*}^{n}x_i+O(1).
\end{equation*}
The assumption $n x_n\to x$ gives
$
x_n=\frac{x}{n}+o\left(\frac{1}{n}\right).
$
Consequently,
\[
\sum_{i=n^*}^{n}x_i
=
x\sum_{i=n^*}^{n}\frac{1}{i}
+o(\ln n)
=
x\ln n+o(\ln n).
\]
Hence, we get $
\ln \pi_n=x\ln n+o(\ln n)
$, 
which is equivalent to
\[
\pi_n=n^{x+o(1)}.
\]
It follows immediately that, for every $\varepsilon>0$,
\[
\pi_n=O\left(n^{x+\varepsilon}\right).
\]
Suppose now that the series 
\(
\sum_{n=n^*}^{\infty}
\left(x_n-\frac{x}{n}\right)
\) is convergent. 
Then, by~\eqref{eq-Taylor}, using the convergence of the above series and $\sum_n x_n^2<\infty$, the series
\[
\sum_{n=n^*}^{\infty}
\left(
-\ln(1-x_n)-\frac{x}{n}
\right)
\]
converges. Since
$
(\sum_{i=n^*}^{n}\frac{1}{i}-\ln n)
$
also converges as $n\to\infty$, it follows that
\[
\ln \pi_n-x\ln n
\]
converges to a finite real number. Therefore, there exists
$L\in(0,\infty)$ such that
$
\pi_n\sim L n^x.
$
\end{proof}
\begin{cor}\label{cor-prod}
Let $\eta>-1$ and $\eta\neq 0$, and $p\in[0,1)$. Let $(p_n)$ and $(c_n)$
be sequences of nonnegative real numbers such that 
there exist $\beta>0$ and $\theta<1$ with 
\begin{equation}\label{eq-ass}
i)\quad p_n-p=O\left(n^{-\beta}\right)
\qquad
\hbox{and}
\qquad 
ii)\quad c_n-n p=O\left(n^{\theta}\right).
\end{equation} Set, for $n$ large enough to have $n+\eta c_n>0$, 
\[
x_n
=
\frac{(1+\eta) p_n}{n+\eta c_n}
=
\frac{1}{n}
\frac{(1+\eta) p_n}{1+\eta c_n/n}
\qquad\hbox{and}
\qquad 
x
=
p \frac{(1+\eta)}{1+p\eta}.
\]
Then, for $n^*$ sufficiently large, there exists a constant
$L\in(0,\infty)$ such that
\[
\prod_{i=n^*}^{n}\frac{1}{1-x_i}
\sim L n^x.
\]
Moreover, in the particular case  $c_n=O( \sum_{i=0}^{n-1} p_i)$, we have:
\begin{itemize}
\item when $p=0$, \eqref{eq-ass}i) implies \eqref{eq-ass}ii)
\item when $p>0$, \eqref{eq-ass}i) implies \eqref{eq-ass}ii) provided 
$c_n-\sum_{i=0}^{n-1} p_i=O(n^\alpha)$ with $\alpha<1$. 
\end{itemize}
\end{cor}
\begin{proof}
The assumption on $c_n$ yields
$
\frac{c_n}{n}=p+O\left(n^{-(1-\theta)}\right).
$
Consider the function
\[
h(a,b)=\frac{(1+\eta) a}{1+\eta b}.
\]
Since
the function $h$ is continuously differentiable in a neighborhood of
$(p,p)$, setting  
$\beta^*=\min\{\beta,1-\theta\}>0,$
 we obtain
\begin{align*}
n x_n
&=
h\left(p_n,\frac{c_n}{n}\right)
=
h(p,p)+O\left(n^{-\beta^*}\right)\\
&=
p\frac{1+\eta}{1+\eta p}
+O\left(n^{-{\beta^*}}\right)
=
x+O\left(n^{-{\beta^*}}\right).
\end{align*}
Therefore,
$
x_n
=
\frac{x}{n}
+
O\left(n^{-(1+\beta^*)}\right)
$
and so 
$$
\sum_n \Big|x_n-\frac{x}{n}\Big|<\infty.
$$
Lemma~\ref{lemma-prod} now implies that there exists a constant $L\in(0,\infty)$
such that
\[
\prod_{i=n^*}^{n}\frac{1}{1-x_i}
\sim L n^x.
\]
If $c_n=O(\sum_{i=0}^{n-1} p_i)$ and $p=0$, then \eqref{eq-ass}$i)$ implies 
$c_n=O(n^{\theta})$ with $\theta\geq \max\{(1-\beta),0\}$ when $\beta\neq 1$ and with $\theta>0$ when $\beta=1$ and so the above required assumption~\eqref{eq-ass}$ii)$ is verified taking $ \max\{(1-\beta),0\}\leq \theta<1$ when $\beta\neq 1$ and $0<\theta<1$ when $\beta=1$.  
If $p>0$ and $c_n-\sum_{i=0}^{n-1} p_i=O(n^\alpha)$ with $\alpha<1$, then \eqref{eq-ass}$i)$ implies $c_n=O(n^\alpha)+ np + O(\sum_{k=1}^n 1/k^\beta)=n p+O(n^\theta)$ with 
 a suitable $\theta<1$.  
\end{proof}
}
\section{Theory agrees with practice}
\subsection{Asymptotic behaviors and parameter estimation}
{This section provides a brief selection of the simulations conducted to  validate and complement the theoretical results presented in this work.  
For both the scenarios described in Examples~\ref{ex-1} and \ref{ex-2}, we perform $10$  independent simulations of the FAQ urn process, each run up to time $2\cdot 10^5$.
The parameter settings are the following: }
\begin{itemize}
    \item[(1)] Example 3.1: $F(x)=\rho x+\widetilde{\rho}$ with $\rho=1$,   
    $\widetilde{\rho}=0$ or $\widetilde{\rho}=1$ and $p_n=0.3$;
    \item[(2)] Example 3.2: $F(x)=\rho x + \widetilde{\rho}$  with $\rho=1$, $\widetilde{\rho} =0$ and $p_n=n^{-0.3}$.
\end{itemize}
For each scenario, {we plot the number of colors, $C_n$,
and the number of draws of each color, $K_n(c)$, over time, as well as the frequency--rank curve $z(r)$.}
We do not report the color--frequency distribution $q(k)$, since it is always closely related to the frequency--rank curve $z(r)$.
Figures~\ref{fig:Example_1} and \ref{fig:Example_2} display the simulation results.
In each figure, the black line represents the fitted curve obtained via linear regression, which closely matches the simulated curves. In order to highlight this match, in Table~\ref{Table simulation data} we compare the theoretical slopes proven in this article with the empirical slopes obtained by fitting the simulated data using linear regressions. Specifically, the regressions reported in Table~\ref{Table simulation data} are referred to the following asymptotic (for large $n$) relations:
\begin{itemize}
    \item[(i)] (\emph{Number of colors over time}) $\log_{10}(C_n)$ increases linearly with $\log_{10}(n)$;
    \item[(ii)] (\emph{Number of draws of a given color over time}) $\log_{10}(K_n(c))$ increases linearly with $\log_{10}(n)$;
    \item[(iii)] (\emph{Frequency--rank curve}) $\log_{10}(z(r))$ decreases linearly with $\log_{10}(r)$.
\end{itemize}

As shown in Table~\ref{Table simulation data}, all estimates obtained in simulations are in agreement with their theoretical counterparts.

\renewcommand{\arraystretch}{1.5}
\begin{table}[htbp!]
\caption{Comparison between the empirical and theoretical slopes for the linear regressions (i)-(ii)-(iii) in the frameworks of 
Examples~\ref{ex-1} and \ref{ex-2}.}
\label{Table simulation data}
\centering
\small
\begin{tabular}{|l|l|r|r|r|}
\hline
Regression & slope & Ex.~3.1 ($\widetilde{\rho}=0$) & Ex.~3.1 ($\widetilde{\rho}=1$) & Ex.~3.2\\
\hline
\multirow{2}{*}{\shortstack{$\log_{10}(C_n)$\\ vs\\ $\log_{10}(n)$}}
    & emp   & $1.002$ & $1.002$ & $0.700$\\
    & theo  & $1.000$ & $1.000$ & $\theta = 0.700$ \\
\hline
\multirow{2}{*}{\shortstack{$\log_{10}(K_n(c))$\\ vs\\ $\log_{10}(n)$}}
    & emp   & $0.697$ & $0.540$ & $0.960$\\
    & theo  & $1-p=0.700$ & $\frac{1-p}{1+p} = 0.538$ & $1.000$ \\
\hline
\multirow{2}{*}{\shortstack{$\log_{10}(z(r))$\\ vs\\ $\log_{10}(r)$}}
    & emp   & $-0.709$ & $-0.572$ & $-1.307$\\
    & theo  & $-(1-p) = -0.700$ & $-\frac{1-p}{1+p} = -0.538$ & $-\frac{1}{\theta} = -1.429$ \\
\hline
\end{tabular}
\end{table}

\begin{figure}[!htbp]
    \centering
    \includegraphics[scale=0.3]{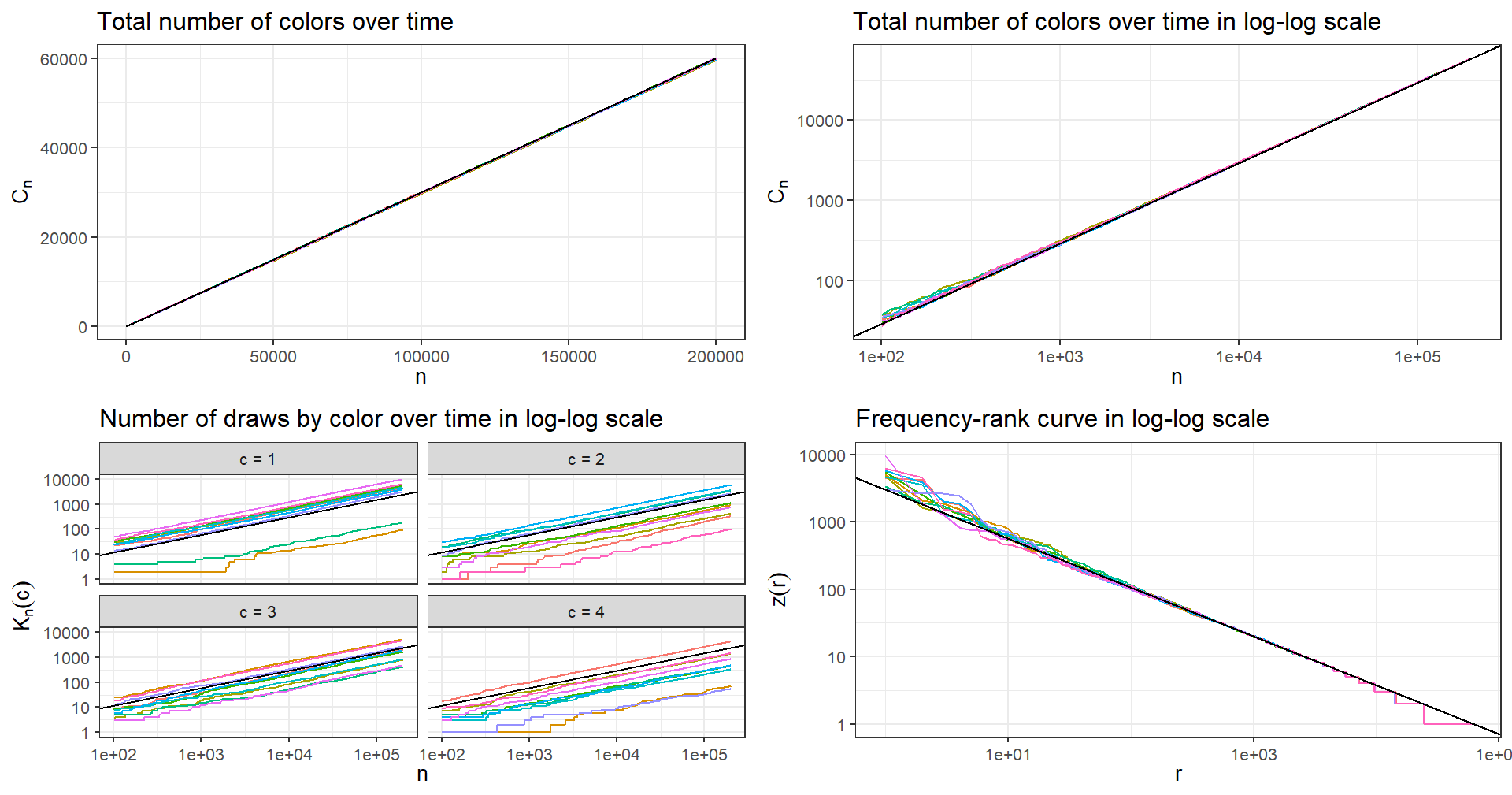}
    \caption{Example 3.1 with $F(x)=x$ and $p_n=0.3$: overview of the simulation results.}
    \label{fig:Example_1}
\end{figure}

\begin{figure}[!htbp]
    \centering
    \includegraphics[scale=0.3]{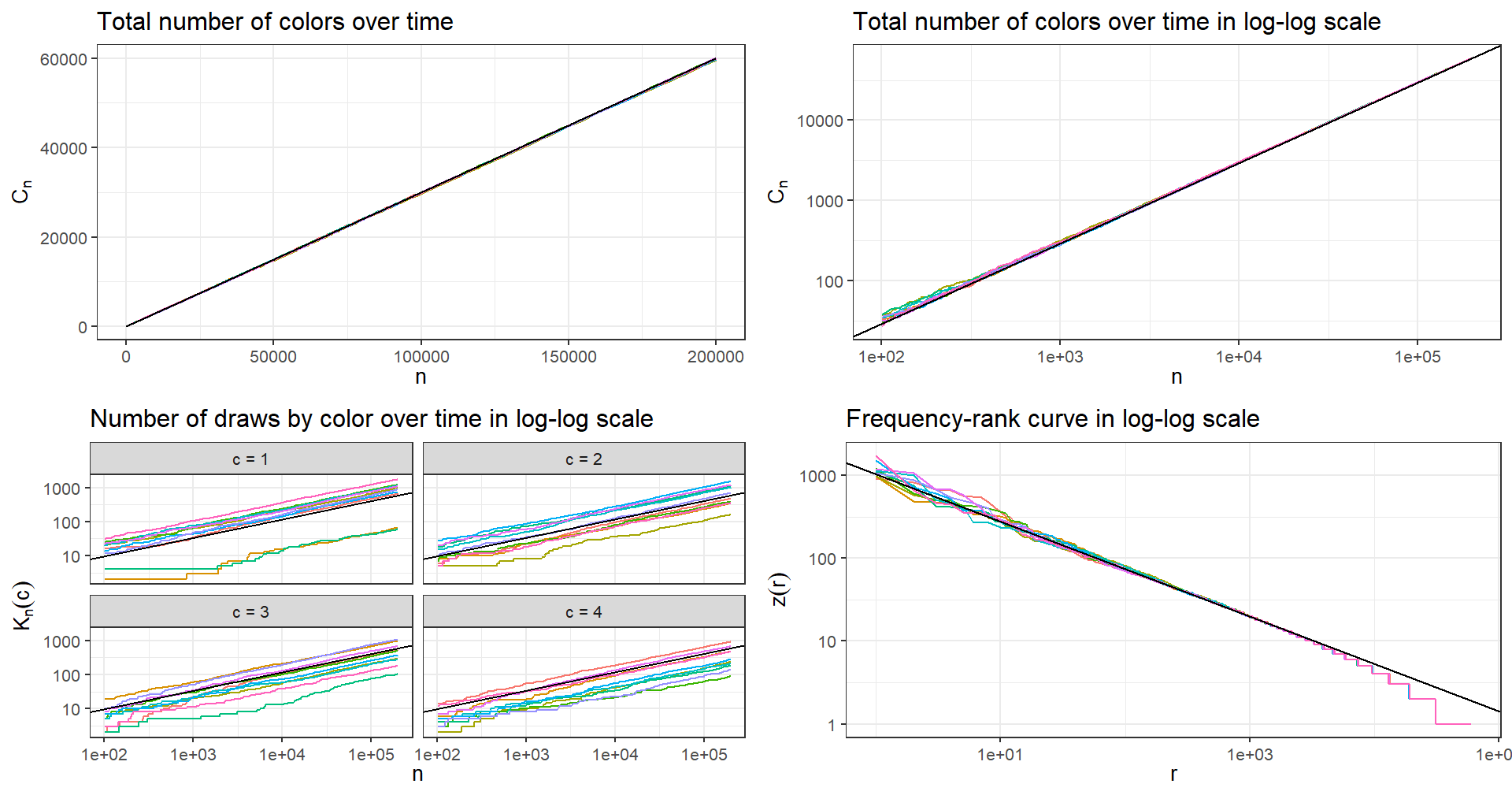}
    \caption{Example 3.1 with $F(x)=x+1$ and $p_n=0.3$: overview of the simulation results.}
    \label{fig:Example_1_remark}
\end{figure}

\begin{figure}[!htbp]
    \centering
    \includegraphics[scale=0.25]{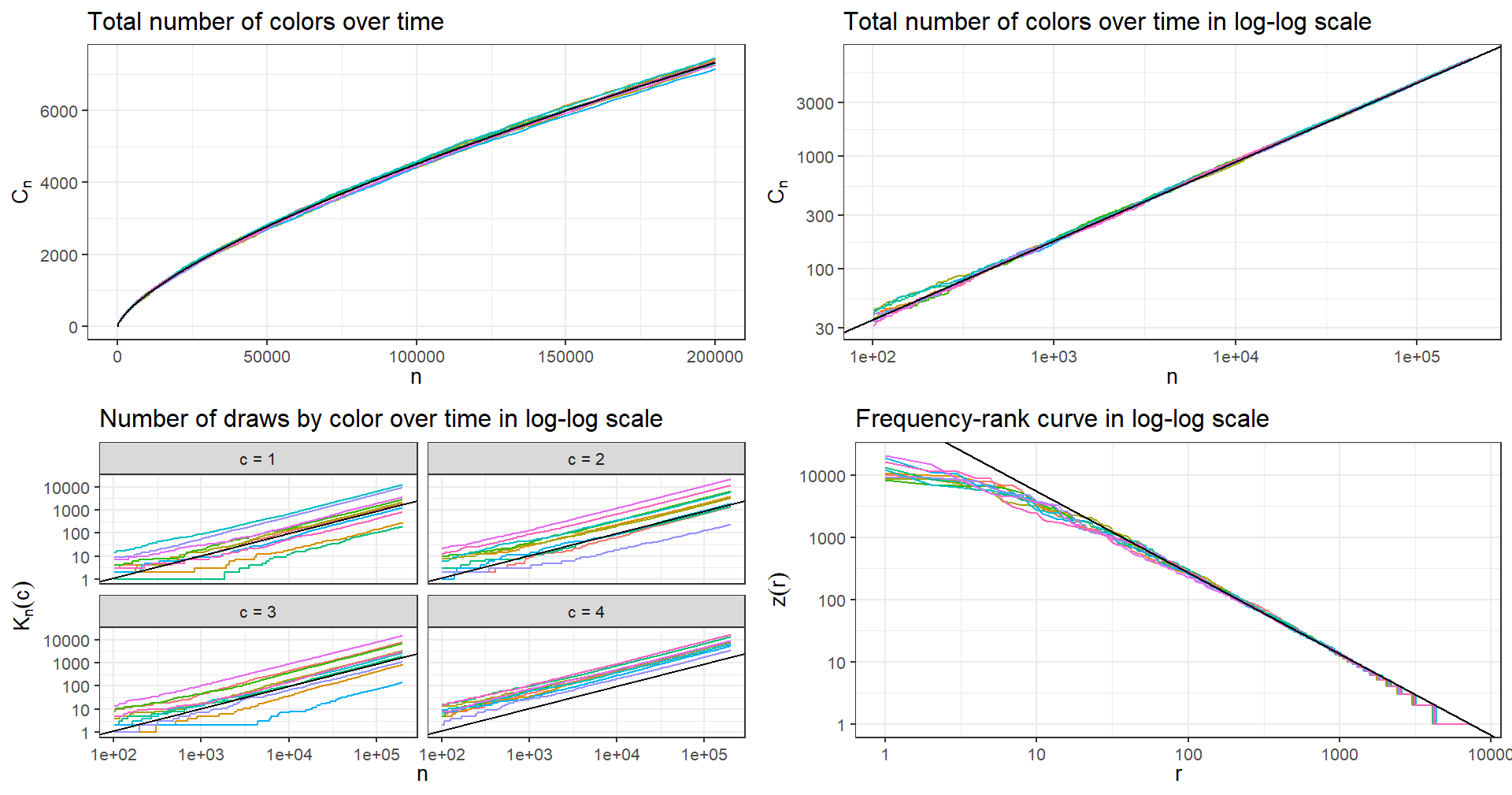}
    \caption{Example 3.2 with $F(x)=x$ and $p_n=n^{-0.3}$: overview of the simulation results.}
    \label{fig:Example_2}
\end{figure}

In addition, we took a look at the performance of the estimator of the ratio 
$\widetilde{\rho}/\rho$, defined in~\eqref{eq:ratio_rhotilderho_estimator}, within the 
framework of Example~\ref{ex-1}, i.e., $F(x)=\rho x + \widetilde{\rho}$ and $p_n = p$. 
Specifically, we set $p = 0.3$, $\rho = 1$, and considered 
two scenarios with 
$\widetilde{\rho} = 0$ and $=1$, respectively. For each scenario, we ran 
10 independent simulations of the FAQ urn process, each up to time $2 \cdot 10^5$. 
The parameter $p$ was estimated by $\widehat{p}$, defined as 10 raised to the power of the intercept obtained from the log--log regression (i) on~$C_n$, obtaining, for both the two  
scenarios, the estimate $\widehat{p}=0.293$. The exponent $\delta$ of the power-growth of $K_{n,c}$ was estimated as the empirical slope  $\widehat{\delta}$ in the log--log regression (ii). Using $\widehat{p}$ and $\widehat{\delta}$, we computed the value of the estimator $\widehat{(\widetilde{\rho}/\rho)}$ defined  in~\eqref{eq:ratio_rhotilderho_estimator}, obtaining, for the two  
scenarios, the estimates $0.049$ and $1.056$, respectively.
Finally, for the case $\widetilde{\rho} = 1$, we used this last estimate to recompute the slope $-\widehat{a}$ of the rank-frequency curve by replacing $\log_{10}(z(r))$ with $\log_{10}\big(z(r) + \widehat{\widetilde{\rho}/\rho}\,\big)$, obtaining the value  $-0.556$. This slightly improves the estimate $-0.572$ reported in Table~\ref{Table simulation data}.
\subsection{Short-term predictions}
{We compare the asymptotics 
in Theorem~\ref{Thm:Kcn} with the combinatorial analysis in Subsection~\ref{Sec:freq},
using $n=500$, $F(x) =x$ (as in Simon's urn). The intent is to show that 
the asymptotics are well aligned with the exact combinatorial analysis.

Let us take the number $K_{500,2}$ as a benchmark
under some selected values of $p$.}  
With $p = 0.95$,we get
$$\frac {K_{500,2}} {(500)^{0.05}} = 0.977022.$$ 
{This is some confirmation that the scaling factor in Theorem~\ref{Thm:Kcn} is plausible, and identifies $K(2)$. The predicted value in Corollary~\ref{Cor:Xnge1} is $0.9770681002\ldots$ .
The asymptotic value is very close to the exact, indicating a robust rate of convergence.}
Note how both
theory and practice produce a low expectation of the number of balls of color 2, as $p$ is quite high with a propensity to produce new balls at each epoch, rather than pick from competing existing colors. 

At the other end of the spectrum, with $p = 0.1$,
$$\frac {K_{500,2}} {(500)^{0.9}} = 0.203328\ldots $$ 
while the predicted value in Corollary~\ref{Cor:Xnge1} is $0.203346\ldots$ .


\begin{thebibliography}{99}
 \bibitem{ale-cri}         
     Aletti, G. and Crimaldi, I.\ (2021). Twitter as an innovation process with damping effect.
      {\em Scientific Reports} {\bf 11}, 21243.
\bibitem{ale-cri-ghi}   
Aletti, G., Crimaldi, I.~Ghiglietti, A.\ (2023). 
     Interacting innovation processes. 
    {\em Scientific Reports} {\bf 13}, 17187.
 \bibitem{Barbour}  
     Barbour, A.\ and Hall, P.\ (1984).
     On the rate of Poisson convergence. 
      {\em Mathematical Proceedings of Cambridge Philosophical Society}
      {\bf 95}, 473--480.
\bibitem{Holst}  Barbour, A., Holst, L.\ and Janson, S.\  (1992).
   {\em Poisson Approximation}. 
    Clarendon Press, Oxford, UK.   
\bibitem{Campos} 
Campos, J. A., Otegi A., Soroa A., Deriu J., Cieliebak M.\  and Agirre, E.\ (2020). 
DoQA -  Accessing Domain-Specific FAQs via Conversational QA. 
{\em Proceedings of the 58th Annual Meeting of the Association for Computational Linguistics}, 7302--7314. 
\bibitem{Chung}
    Chung F., Handjani S., and Jungreis D.\ (2003).
    Generalizations of P\'olya’s urn Problem. 
    {\em Annals of Combinatorics} {\bf 7},  141--153.
\bibitem{Collevecchio}
             Collevecchio, A., Cotar, C., and Li Calzi, M.\ (2013). 
             On a preferential attachment and generalized P\'olya's urn model
              {\em Annals of Applied Probability} 
              {\bf  23}, 1219--1253. 
\bibitem{DeBruyn}
De Bruyn, M.,  Lotfi, E., Buhmann, J.\  and Daelemans, W.\ (2021). 
MFAQ: a Multilingual {FAQ} Dataset. {\em Proceedings of the 3rd Workshop on Machine Reading for Question Answering}, 1--13.
\bibitem{Din}
Dinzinger, M., Caspari, L., Ghosh Dastidar, K., Mitrovi\'{c}, J.\ and Granitzer, M.\ (2025).  WebFAQ: A Multilingual Collection of Natural Q\&A Datasets for Dense Retrieval. 
{\em Proceedings of the 48th International ACM SIGIR Conference on Research and Development in Information Retrieval}, 3802–-3811. 
    \bibitem{Karr}
     Karr, A.:
     {\em Probability}. Springer-Verlag, New York, 1993.
     \bibitem {Bala}
        Kotz, S.\ and Balakrishnan, N.\ (1997). 
        Advances in urn models during the past two decades.
        In {\em Advances in Combinatorial Methods and Applications
        to Probability and Statistics} 
        {\bf 49}, 203--257. Birkh\"auser, Boston.
\bibitem{Mah2009}
        Mahmoud, H.\ (2009). 
       {\em \polya\ Urn Models}. 
       CRC Press, Boca Raton, FL.
       {
\bibitem{RobSie1971} Robbins, H. and Siegmund, D.\ (1971). 
A convergence theorem for nonnegative almost supermartingales and some applications. 
  {\em Optimizing Methods in Statistics}, 233--257.    Academic Press, New York.
  }
\bibitem{Simon} 
         Simon, H.\ (1955). 
        On a class of skew distribution functions.
        {\em Biometrika} {\bf 42}, 
        425--440.
        \bibitem{Sumikawa}
Sumikawa, Y., Fujiyoshi, M., Hatakeyama, H.\ and Nagai, M.\ (2019). 
An FAQ dataset for E-learning system used on a Japanese University. 
{\em Data in  Brief} {\bf 25}, 104001.
\bibitem{Tria1}
           Tria, F., Crimaldi, I. Aletti, G.\ and Servedio, V.\ (2020).
           Taylor’s law in innovation processes.
           {\em Entropy} 
           {\bf 22}, 573.
\bibitem{Tria2}
    Tria, F., Loreto, V.\ and Servedio, V.\ (2018). 
    Zipf’s, Heaps’ and Taylor’s laws are determined by
    the expansion into the adjacent possible,
    {\em Entropy} 
    {\bf 20}, 752.
\bibitem{Tria3}
         Tria, F., Loreto, V., Servedio, V.\ and   Strogatz, S.\ (2014).
         The dynamics of correlated novelties.
          {\em Scientific Reports} 
          {\bf 4}, 5890.
\bibitem{wan-mcauley}
        Wan, M. and McAuley, J.\ (2016).
        Modeling ambiguity, subjectivity, and diverging viewpoints in opinion question answering systems.
        {\em Proceedings of the IEEE International Conference on Data Mining (ICDM)}, 1009--1014.
\bibitem{Williams} 
        Williams, D.\ (1991).
        {\em Probability with Martingales}. 
        Cambridge Mathematical Textbooks. 
        Cambridge University Press, Cambridge, UK.
\bibitem{Zanette}
Zanette D. and Montemurro, M.~(2005). 
Dynamics of text generation with realistic Zipf's distribution.
{\em	Journal of Quantitative Linguistics }  {\bf 12}, 29--40.
\bibitem{Zhang}
Zhang, X. F.,  Sun, H., Yue, X., Lin, S.\  an Sun, H.\ (2021). 
 COUGH: A Challenge Dataset and Models for {COVID}-19 {FAQ} Retrieval. {\em Proceedings of the 2021 Conference on Empirical Methods in Natural Language Processing}, 3759--3769. 
    
\end{thebibliography}
\end{document}